\documentclass[11pt]{article}

\usepackage{amsmath,amssymb,amsthm,mathrsfs}
\usepackage{xcolor}
\usepackage{fullpage}
\usepackage{hyperref}
\hypersetup{
  hidelinks,
  pdftitle={A Logarithmic Fluctuation Hierarchy for Sequential Interacting Diffusions},
  pdfauthor={Zhenfu Wang and Xianliang Zhao},
  pdfkeywords={sequential interacting diffusions, non-exchangeable particle systems, Gaussian fluctuations, central limit theorem, logarithmic fluctuation hierarchy}
}

\newtheorem{theorem}{Theorem}[section]
\newtheorem{corollary}[theorem]{Corollary}
\newtheorem{proposition}[theorem]{Proposition}
\newtheorem{lemma}[theorem]{Lemma}
\newtheorem{remark}[theorem]{Remark}
\newtheorem*{assumptionA}{Assumption A}
\numberwithin{equation}{section}

\newcommand{\R}{\mathbb R}
\newcommand{\Law}{\mathcal L}
\newcommand{\dd}{\,\mathrm d}
\newcommand{\LawConv}{\Rightarrow}
\title{A Logarithmic Fluctuation Hierarchy for Sequential Interacting Diffusions}

\author{Zhenfu Wang and Xianliang Zhao}

\date{}

\begin{document}
\maketitle

\begin{abstract}
We study Gaussian fluctuations for a lower-triangular system of
interacting diffusions in which particle \(i\) interacts only with its
predecessors. Although the empirical measure of this system converges to the
same McKean--Vlasov limit as in the corresponding exchangeable mean-field
system, the sequential structure remains visible at the \(N^{-1/2}\)
fluctuation scale.
We introduce the logarithmically weighted fluctuation fields
\[
Y_t^{N,n}
=
\frac1{\sqrt N}
\sum_{i=1}^N
\frac{\bigl(\log(N/i)\bigr)^n}{n!}
(\delta_{X_t^i}-\bar\rho_t),
\qquad n\ge 0,
\]
and prove joint convergence of the entire family in a countable product of
weighted negative Sobolev path spaces to the unique probabilistically strong
solution of a linear hierarchy in which \(Y^n\) couples to \(Y^{n+1}\). In
particular, the limit of the empirical fluctuation field
\(\sqrt N(\mu_t^N-\bar\rho_t)=Y_t^{N,0}\) is not governed by the closed
fluctuation SPDE arising in the classical exchangeable case. The proof
combines conditional-measure replacement, deterministic estimates for the
logarithmic weights, a martingale argument, and a weighted Volterra estimate.
\end{abstract}


\section{Introduction and main result}

Interacting particle systems approximate McKean--Vlasov diffusions and their
nonlinear Fokker--Planck equations at two basic scales. At first order, a law
of large numbers gives convergence of the empirical measure to the
deterministic solution of the limiting nonlinear equation. This is the
classical mean-field limit picture; see, for instance,
McKean~\cite{mckeanpropagation}, Sznitman~\cite{sznitman1991topics}, and the
surveys and quantitative works
\cite{fournier2014propagation,jabin2014review,jabin2018quantitative,lacker2021hierarchies,serfaty2020mean}.
At second order, the central limit correction
\[
  \mu^N \approx \bar\rho + N^{-1/2}\eta,
\]
describes the statistical error around that limit. In the classical
exchangeable model, where each particle interacts symmetrically with the
empirical measure of the whole population, this correction is governed by a
closed linear equation for the single averaged fluctuation field.
This theory is usually formulated through linear martingale
problems or stochastic evolution equations; see, for instance,
\cite{mckean1975fluctuations,meleard1996asymptotic,fernandez1997hilbertian,wang2023gaussian}.

The question addressed here is whether this classical fluctuation equation is
determined by the McKean--Vlasov law of large numbers. We show that it is not:
a sequential system and an exchangeable system may have the same deterministic
limit but different Gaussian fluctuations.

We study how the fluctuation picture changes when exchangeability is replaced
by a sequential order. The model is the lower-triangular particle system
\begin{equation}\label{eq:seq-sde}
\begin{aligned}
  \dd X_t^1
  &= b(t,X_t^1)\,\dd t+\dd B_t^1,\\
  \dd X_t^i
  &=
  \Big(b(t,X_t^i)+(K*\mu_t^{i-1})(X_t^i)\Big)\,\dd t
  +\dd B_t^i,\qquad i\ge2,
\end{aligned}
\end{equation}
where
\[
  \mu_t^{i-1}:=\frac1{i-1}\sum_{j<i}\delta_{X_t^j}.
\]
The family \((B^i)_{i\ge1}\) consists of i.i.d.\ standard \(d\)-dimensional
Brownian motions. Particle \(i\) interacts only with its predecessors. We write
\[
  \mu_t^N:=\frac1N\sum_{i=1}^N\delta_{X_t^i}
\]
for the global empirical measure. The lower-triangular construction is
naturally online: after \((X^1,\ldots,X^N)\) has been generated, adding
particle \(X^{N+1}\) does not change any of the preceding trajectories.
The system is directed and non-exchangeable, but it approximates the same
McKean--Vlasov diffusion as the exchangeable model:
\begin{equation}\label{eq:limit-sde}
  \dd\bar X_t
  =
  \Big(b(t,\bar X_t)+(K*\bar\rho_t)(\bar X_t)\Big)\,\dd t
  +\dd B_t,
  \qquad \bar\rho_t=\Law(\bar X_t).
\end{equation}

Sequential particle approximations and related non-exchangeable diffusions
were studied by Du--Jiang--Li~\cite{du2023sequential} and
Lacker--Yeung--Zhou~\cite{lacker2024quantitative}. Sharp quantitative
convergence of the lower-triangular system \eqref{eq:seq-sde} to the
McKean--Vlasov limit \eqref{eq:limit-sde} was proved in the authors' companion
work \cite{WangZhaoSequentialEntropy}, using incremental relative entropy and
conditional measure replacement. Thus, at the level of the law of large
numbers, the sequential structure is invisible: the empirical measure still
converges to \(\bar\rho_t\).

At the fluctuation scale, however, this structure becomes visible when the
predecessor sums are reversed. The resulting coefficient of particle \(j\) is
its cumulative influence weight on all successors, namely the sum of the
\(j\)-th column of the interaction matrix. In the balanced non-exchangeable
regime considered by Shkolnikov and
Yeung~\cite{shkolnikovyeung2026universal}, these cumulative weights are
asymptotically constant and the classical mean-field fluctuation equation is
recovered. In the sequential system they instead have the nonconstant profile
\(\log(1/u)\). This is the source of the nonclassical fluctuation limit; a
detailed comparison is given after
Theorem~\ref{thm:log-hierarchy}.

Our main result identifies the correct functional limit in this unbalanced
regime. The averaged fluctuation field is not closed: it couples to a first
logarithmically weighted field, which couples to a second one, and so on. We
prove convergence of the entire countable sequence in the product topology,
identify the covariances of its initial fields and martingale noises, and
prove strong well-posedness of the limiting linear hierarchy. The logarithmic
weights are unbounded near the earliest particle labels, and the hierarchy
never closes at a finite level; controlling these two features is the main
analytical task.

\begin{assumptionA}\phantomsection\label{ass:A}
Let \(T>0\) and \(d\ge1\).
The following assumptions are used throughout the paper.
\begin{enumerate}
\item \(b\in C_b^\infty([0,T]\times\R^d;\R^d)\).
\item The interaction is convolutional: \((K*\nu)(x)=\int_{\R^d}K(x-y)\nu(\dd y)\).
\item \(K\in C_b^\infty(\R^d;\R^d)\cap L^1(\R^d;\R^d)\).
\item The initial variables \(X_0^i\) are i.i.d.\ with common law
\(\bar\rho_0\in\mathcal P(\R^d)\) and are independent of the Brownian motions
\((B^i)_{i\ge1}\).
\end{enumerate}
\end{assumptionA}

\begin{remark}\label{rem:strong-assumptions}
Assumption~A is imposed for convenience rather than
optimality. It allows us to apply directly the quantitative results of
\cite{WangZhaoSequentialEntropy} and to carry out the fluctuation analysis
without developing a separate regularity theory for the McKean--Vlasov limit.
Some regularity is still needed to make the linearized limiting equation
well-defined and well-posed. The coefficient assumptions could be relaxed
substantially; in particular, rough or singular interactions could be treated
by requiring sufficient regularity of the McKean--Vlasov limit to compensate
for the loss of regularity of the kernel. This trade-off is used in
\cite{wang2023gaussian}. We do not pursue such extensions here.
\end{remark}

\theoremstyle{definition}
\newtheorem{definition}[theorem]{Definition}
\theoremstyle{plain}

\paragraph{Derivation of the logarithmic hierarchy.}
Before stating the main theorem, we derive at the
finite-particle level the mechanism that determines the limiting system and
introduces the operators appearing in its definition.
Sections~\ref{sec:identification-inputs} and
\ref{sec:limit-identification} justify the limiting steps.
We begin with the averaged fluctuation measure process
\[
  Y_t^{N,0}
  :=
  \sqrt N(\mu_t^N-\bar\rho_t)
  =
  \frac1{\sqrt N}\sum_{i=1}^N(\delta_{X_t^i}-\bar\rho_t).
\]
Let \(L_t\) be the backward transport-diffusion operator along the
McKean--Vlasov limit:
\[
  L_t\varphi
  =
  \frac12\Delta\varphi
  +
  \big(b(t,\cdot)+K*\bar\rho_t\big)\cdot\nabla\varphi.
\]
For every \(\varphi\in C_c^\infty(\R^d)\), It\^o's formula and the decomposition
\(\delta_{X_t^i}=\bar\rho_t+(\delta_{X_t^i}-\bar\rho_t)\) give
\[
  \begin{aligned}
  \dd\langle Y_t^{N,0},\varphi\rangle
  ={}&
  \langle Y_t^{N,0},L_t\varphi\rangle\,\dd t
  +I_{t,0}^{N,0}(\varphi)\,\dd t
  +\mathcal R_t^{N,0}(\varphi)\,\dd t
  \\
  &-
  \frac1{\sqrt N}
  \nabla\varphi(X_t^1)\cdot(K*\bar\rho_t)(X_t^1)\,\dd t
  +\dd M_t^{N,0}(\varphi),
  \end{aligned}
\]
where \(I_{t,0}^{N,0}\) is the part linear in the predecessor
fluctuation and \(\mathcal R_t^{N,0}\) is the quadratic remainder:
\[
  \begin{aligned}
  I_{t,0}^{N,0}(\varphi)
  &:={}
  \frac1{\sqrt N}
  \sum_{i=2}^N
  \left\langle
    \bar\rho_t,
    (K*(\mu_t^{i-1}-\bar\rho_t))\cdot\nabla\varphi
  \right\rangle,
  \\
  \mathcal R_t^{N,0}(\varphi)
  &:={}
  \frac1{\sqrt N}
  \sum_{i=2}^N
  \left\langle
    \delta_{X_t^i}-\bar\rho_t,
    (K*(\mu_t^{i-1}-\bar\rho_t))\cdot\nabla\varphi
  \right\rangle,
  \end{aligned}
\]
and \(M^{N,0}(\varphi)\) denotes the continuous martingale
arising from the Brownian terms.
The particle-\(1\) term and the time integral of
\(\mathcal R^{N,0}\) vanish as \(N\to\infty\), the latter by
Section~\ref{sec:conditional-replacement}; hence the sequential effect enters
through \(I_{t,0}^{N,0}\).

Indeed, using
\[
  \mu_t^{i-1}-\bar\rho_t
  =
  \frac1{i-1}
  \sum_{j<i}(\delta_{X_t^j}-\bar\rho_t)
\]
and reversing the order of summation, we obtain
\begin{equation}\label{eq:intro-I00-reversed-sum}
  \begin{aligned}
  I_{t,0}^{N,0}(\varphi)
  ={}&
  \frac1{\sqrt N}
  \sum_{j=1}^{N-1}
  \left(
    \sum_{i=j+1}^N\frac1{i-1}
  \right)
  \\
  &\qquad\times
  \left\langle
    \bar\rho_t,
    (K*(\delta_{X_t^j}-\bar\rho_t))\cdot\nabla\varphi
  \right\rangle.
  \end{aligned}
\end{equation}
Writing \(h_0\equiv1\), particle \(j\) receives the weight
\[
  (\mathsf T_N h_0)(j/N)
  :=
  \sum_{i=j+1}^N\frac1{i-1},
  \qquad 1\le j\le N-1,
  \qquad
  (\mathsf T_N h_0)(1):=0.
\]
When \(j/N\to u\in(0,1]\),
\[
  (\mathsf T_N h_0)(j/N)
  \longrightarrow
  \int_u^1\frac{\dd v}{v}
  =
  \log(1/u).
\]
The first logarithmic weight is therefore
\[
  h_1(u):=\log(1/u)
\]
and the corresponding fluctuation measure process
\[
  Y_t^{N,1}
  :=
  \frac1{\sqrt N}
  \sum_{i=1}^N
  h_1(i/N)(\delta_{X_t^i}-\bar\rho_t).
\]
The limiting equation for the averaged fluctuation then has the form
\[
  \begin{aligned}
  \langle Y_t^0,\varphi\rangle
  ={}&
  \langle Y_0^0,\varphi\rangle
  +
  \int_0^t\langle Y_s^0,L_s\varphi\rangle\,\dd s
  \\
  &+
  \int_0^t
  \left\langle
    \bar\rho_s,
    (K*Y_s^1)\cdot\nabla\varphi
  \right\rangle\,\dd s
  +M_t^0(\varphi).
  \end{aligned}
\]
Here \(M^0\) is the limiting martingale corresponding to \(M^{N,0}\). A
classical closed fluctuation equation would contain \(Y^0\) in the last drift
term; the sequential equation contains \(Y^1\). Thus the equation for
\(Y^0\) is already not closed.

Applying the same calculation to \(Y^{N,1}\) replaces \(h_1\) by
\(\mathsf T_N h_1\), and \(\mathsf T_N h_1\to \mathsf T h_1=h_2\), where
\(h_2(u)=\bigl(\log(1/u)\bigr)^2/2\). Hence the limiting equation for \(Y^1\) depends on
\(Y^2\).

For general \(n\), predecessor summation produces the discrete weight
\[
  (\mathsf T_N h_n)(j/N)
  :=
  \sum_{i=j+1}^N\frac{h_n(i/N)}{i-1},
  \qquad 1\le j\le N-1,
  \qquad
  (\mathsf T_N h_n)(1):=0.
\]
The limiting operator is
\[
  (\mathsf T f)(u)
  :=
  \int_u^1\frac{f(v)}{v}\,\dd v.
\]
Starting from \(h_0\equiv1\) and defining \(h_{n+1}=\mathsf T h_n\), we obtain
\[
  h_n(u)=\frac{\bigl(\log(1/u)\bigr)^n}{n!},
  \qquad n\ge0.
\]
We accordingly set
\[
  Y_t^{N,n}
  :=
  \frac1{\sqrt N}
  \sum_{i=1}^N
  h_n(i/N)(\delta_{X_t^i}-\bar\rho_t).
\]
Lemma~\ref{lem:TN-error-summability} proves that
\(\mathsf T_N h_n\to h_{n+1}\). Therefore the equation for \(Y^n\) depends on
\(Y^{n+1}\) for every \(n\ge0\): predecessor summation moves one step forward
in the logarithmic index, and no finite collection of levels is closed.

We now define the limiting logarithmic hierarchy. Fix \(\beta>d/2+4\),
\(m>d/2+1\), and
\(\beta_\ast\) such that
\[
\frac d2+2<\beta_\ast<\beta-2.
\]
Let \(H^{-\beta,-m}\) be the weighted negative Sobolev space defined in
Subsection~\ref{subsec:notation}, and set
\(E:=C([0,T];H^{-\beta,-m})\).

\begin{definition}[Probabilistically weak and strong solutions]
\label{def:log-hierarchy-solution}
We call
\[
(Y,\mathcal M)
=
\bigl((Y^n)_{n\ge0},(\mathcal M^n)_{n\ge0}\bigr)
\]
a probabilistically weak solution of the limiting logarithmic hierarchy if
there exists a stochastic basis
\[
(\Omega,\mathcal F,(\mathcal F_t)_{t\in[0,T]},\mathbb P)
\]
such that, on this basis, \(Y\) and \(\mathcal M\) have the following
properties.
\begin{enumerate}
\item[(i)] For every \(n\ge0\), \(Y^n\) is a continuous
\((\mathcal F_t)\)-adapted process with values in \(H^{-\beta,-m}\) and has
weakly continuous paths in \(H^{-\beta_\ast}\). Moreover, for some
\(q\in(0,1/4)\),
\begin{equation}\label{eq:solution-class-summability}
\sum_{n=0}^\infty
q^n
\mathbb E\sup_{t\le T}\|Y_t^n\|_{H^{-\beta_\ast}}^2
<\infty.
\end{equation}

\item[(ii)] For every \(n\ge0\), \(\mathcal M^n\) is a continuous
\((\mathcal F_t)\)-adapted process with values in \(H^{-\beta,-m}\), and
\(\mathcal M\), viewed as an \(E^{\mathbb N_0}\)-valued random element, is
centered Gaussian. For every \(n\ge0\) and
\(\varphi\in C_c^\infty(\R^d)\),
\[
M_t^n(\varphi):=\langle\mathcal M_t^n,\varphi\rangle
\]
is a continuous \((\mathcal F_t)\)-martingale, and
\begin{equation}\label{eq:limit-martingale-covariance}
\langle M^n(\varphi),M^k(\psi)\rangle_t
=
\frac{(n+k)!}{n!\,k!}
\int_0^t
\langle
\bar\rho_s,
\nabla\varphi\cdot\nabla\psi
\rangle\,\dd s
\end{equation}
for all \(n,k\ge0\), \(t\in[0,T]\), and
\(\varphi,\psi\in C_c^\infty(\R^d)\).

\item[(iii)] For every \(n\ge0\) and
\(\varphi\in C_c^\infty(\R^d)\), the following identity holds
\(\mathbb P\)-almost surely for all \(t\in[0,T]\):
\begin{align}
\langle Y_t^n,\varphi\rangle
={}&
\langle Y_0^n,\varphi\rangle
+
\int_0^t
\langle Y_s^n,L_s\varphi\rangle\,\dd s\notag\\
&+
\int_0^t
\left\langle
\bar\rho_s,
(K*Y_s^{n+1})\cdot\nabla\varphi
\right\rangle\,\dd s
+
\langle\mathcal M_t^n,\varphi\rangle.
\label{eq:limit-hierarchy}
\end{align}
\end{enumerate}

We call \(Y\) a probabilistically strong solution if there exist a stochastic
basis and a process \(\mathcal M\) on it such that \((Y,\mathcal M)\) is a
probabilistically weak solution and \(Y\) is adapted to
the normal filtration generated by
\((Y_0^n)_{n\ge0}\) and \(\mathcal M\).
\end{definition}

\begin{theorem}\label{thm:log-hierarchy}
Under \hyperref[ass:A]{Assumption~A}, as \(N\to\infty\),
\begin{equation}\label{eq:main-product-convergence}
(Y^{N,n})_{n\ge0}
\LawConv
(Y^n)_{n\ge0}
\end{equation}
in \(E^{\mathbb N_0}\), equipped with the countable product topology.
The initial hierarchy \((Y_0^n)_{n\ge0}\) has a centered Gaussian law with
covariance
\begin{equation}\label{eq:limit-initial-covariance}
\mathbb E\!\left[
\langle Y_0^n,\varphi\rangle
\langle Y_0^k,\psi\rangle
\right]
=
\frac{(n+k)!}{n!\,k!}
\left[
\langle\bar\rho_0,\varphi\psi\rangle
-
\langle\bar\rho_0,\varphi\rangle
\langle\bar\rho_0,\psi\rangle
\right]
\end{equation}
for all \(n,k\ge0\) and \(\varphi,\psi\in C_c^\infty(\R^d)\).
For this initial hierarchy, the limiting family \((Y^n)_{n\ge0}\) is the
unique probabilistically strong solution in the sense of
Definition~\ref{def:log-hierarchy-solution}.
\end{theorem}

\paragraph{Comparison with the universal non-exchangeable CLT.}
Recently, Shkolnikov and Yeung
\cite[Theorem~2.4]{shkolnikovyeung2026universal} proved a universal central
limit theorem for a class of non-exchangeable interacting diffusions with
deterministic interaction weights. The empirical measures of these systems
have the same McKean--Vlasov limit as the corresponding exchangeable
mean-field system. In this regime, the
global fluctuation field
\(\sqrt N\bigl(\mu_t^N-\bar\rho_t\bigr)\) converges, in particular, in
\(C([0,T];H^{-\beta})\) to the same Gaussian fluctuation process as in the
classical exchangeable mean-field case, denoted here by \(\eta\):
\[
  \partial_t \eta
  =
  \frac12\Delta \eta
  -\nabla\cdot\bigl(\bar\rho_t(K*\eta)\bigr)
  -\nabla\cdot\bigl(\eta(K*\bar\rho_t)\bigr)
  -\nabla\cdot\bigl(b(t,\cdot)\eta\bigr)
  -\nabla\cdot\bigl(\sqrt{\bar\rho_t}\,\xi\bigr),
  \qquad \eta|_{t=0}=\eta_0,
\]
where \(\xi\) is an \(\R^d\)-valued space--time white noise.

The sequential system considered here has the same McKean--Vlasov limit but a
different fluctuation limit. To compare the two regimes, write \(a_{ij}^N\)
for the weight with which particle \(j\) enters the interaction drift of
particle \(i\), set \(A^N=(a_{ij}^N)\), and define the \(j\)-th column sum
\[
  c_j^N:=\sum_{i=1}^N a_{ij}^N.
\]
Thus \(c_j^N\) is the total weight with which particle \(j\) enters the
interaction drifts of all particles.
The reversed-sum identity \eqref{eq:intro-I00-reversed-sum} for
\(I_{t,0}^{N,0}\) is the special case corresponding to the sequential
weights. For the weight matrix \(A^N\), the same calculation gives the leading
interaction term in the averaged fluctuation equation:
\[
  \frac1{\sqrt N}
  \sum_{j=1}^N c_j^N
  \left\langle
    \bar\rho_t,
    \bigl(K*(\delta_{X_t^j}-\bar\rho_t)\bigr)\cdot\nabla\varphi
  \right\rangle.
\]
Hence \(c_j^N\) is precisely the coefficient with which the fluctuation of
particle \(j\) enters the averaged interaction term.

For systems covered by
\cite[Theorem~2.4]{shkolnikovyeung2026universal}, the row-normalization and
column-balance requirements in
\cite[Assumption~2.1(i)]{shkolnikovyeung2026universal} imply
\[
  \frac1N\sum_{j=1}^N c_j^N=1,
  \qquad
  \frac1N\sum_{j=1}^N(c_j^N-1)^2\longrightarrow0,
\]
so the column sums \(c_j^N\) converge to \(1\) in empirical \(L^2\). The CLT
result in
\cite[Theorem~2.4]{shkolnikovyeung2026universal} justifies replacing
\(c_j^N\) by \(1\) in the leading interaction term, which then converges to
\[
\left\langle
\bar\rho_t,(K*\eta_t)\cdot\nabla\varphi
\right\rangle,
\]
which is the weak form of the term
\(-\nabla\cdot(\bar\rho_t(K*\eta))\) in the SPDE above.

For the sequential weights
\(a_{ij}^N=\mathbf 1_{\{j<i\}}/(i-1)\), however, the reversed-sum identity
\eqref{eq:intro-I00-reversed-sum} gives
\[
  c_j^N
  =
  \sum_{i=j+1}^N\frac1{i-1}
  =
  (\mathsf T_N h_0)(j/N),
  \qquad 1\le j\le N.
\]
Although \(N^{-1}\sum_{j=1}^N c_j^N\to1\), the sequential coefficients
\(c_j^N\) are not asymptotically constant as \(N\to\infty\), in contrast with
the column-balance condition in
\cite[Assumption~2.1(i)]{shkolnikovyeung2026universal}. Indeed, when
\(j/N\to u\in(0,1]\),
\[
  c_j^N=(\mathsf T_N h_0)(j/N)\longrightarrow h_1(u)=\log(1/u),
\]
and
\[
  \frac1N\sum_{j=1}^N(c_j^N-1)^2
  \longrightarrow
  \int_0^1\bigl(\log(1/u)-1\bigr)^2\,\dd u
  =1.
\]
Moreover,
\[
  \max_{1\le j\le N}c_j^N
  =
  c_1^N
  =
  \sum_{i=2}^N(i-1)^{-1}
  \longrightarrow\infty,
\]
which violates the uniform column-sum bound in
\cite[Assumption~2.1(i)]{shkolnikovyeung2026universal}.
The discussion in \cite[Subsection~1.5]{shkolnikovyeung2026universal}
explicitly states that the uniform sequential system considered here is not
covered by \cite[Theorem~2.4]{shkolnikovyeung2026universal}, precisely because
its column sums diverge.

Thus the same coefficient \(c_j^N\) has different limiting profiles in the
two regimes: \(1\) under column balance and \(h_1\) for the sequential
weights. Accordingly, the sequential interaction term converges to
\[
\left\langle
\bar\rho_t,(K*Y_t^1)\cdot\nabla\varphi
\right\rangle,
\]
with \(Y^1\) in place of the unweighted field \(Y^0\), so the equation for
\(Y^0\) couples to \(Y^1\). At level \(n\), predecessor summation produces
\(\mathsf T_N h_n\to\mathsf T h_n=h_{n+1}\), so the equation for \(Y^n\)
couples to \(Y^{n+1}\). This yields the countable logarithmic hierarchy in
Theorem~\ref{thm:log-hierarchy}.

\paragraph{Related literature.}
The fluctuation theory of kinetic and McKean--Vlasov models originates in the
work of McKean~\cite{mckean1975fluctuations},
Tanaka--Hitsuda~\cite{tanaka1981central}, and
Tanaka~\cite{tanaka1984limit}. Subsequent work developed both the analytical
framework and the range of models covered by the theory. For McKean--Vlasov
diffusions, the martingale and Hilbert-space framework was developed by
M\'el\'eard~\cite{meleard1996asymptotic} and
Fern\'andez--M\'el\'eard~\cite{fernandez1997hilbertian}, while related
fluctuation questions for Vlasov-type limits were studied by
Braun--Hepp~\cite{braun1977vlasov} and
Lancellotti~\cite{lancellotti2009fluctuations}. The analysis was later
extended by Budhiraja--Wu~\cite{budhiraja2016some} to weakly interacting
multi-type systems with common factors. In the singular-kernel setting,
Wang--Zhao--Zhu~\cite{wang2023gaussian} combined the martingale approach with
entropy and Donsker--Varadhan estimates. The present proof draws on the same
combination of arguments.

Of particular relevance here is the passage from exchangeable to
non-exchangeable systems. At the law of large numbers level, mean-field limits
without exchangeability have been established for systems on inhomogeneous
and Erd\H{o}s--R\'enyi graphs
\cite{bhamidi2019weakly,delattre2016note,coppini2020law}, for graphon and dense
network models \cite{bayraktar2020graphon,jabin2021mean}, and for sequential
particle systems \cite{du2023sequential,WangZhaoSequentialEntropy}. Moving to
the fluctuation scale, central limit theorems have been obtained for
finite-type systems
\cite{budhiraja2016some}, spatially extended systems
\cite{lucon2016transition}, and interacting diffusions on random graphs
\cite{coppini2023central}. The universal central limit theorem of
Shkolnikov--Yeung~\cite{shkolnikovyeung2026universal} provides the closest
comparison with the present setting. Its column-balance assumptions, however,
exclude the sequential weights considered here.

\paragraph{Proof strategy.}
The proof follows a tightness--identification--uniqueness argument.
Deterministic estimates for the logarithmic weights, combined with weighted
negative-Sobolev bounds, yield tightness of the entire countable family in
path space. To identify a subsequential limit, we use the predecessor-sum
asymptotics \(\mathsf T_N h_n\to h_{n+1}\) and reverse the interaction sums at
level \(n\), which brings \(Y^{n+1}\) into the equation for \(Y^n\). A weighted
central limit theorem and a martingale central limit theorem identify the
initial fields and the martingale noises, respectively, while
conditional-measure replacement shows that the nonlinear remainder vanishes.
These ingredients determine the hierarchy satisfied by every subsequential
limit. Finally, a weighted Volterra estimate for the mild hierarchy gives
pathwise uniqueness. Together with the weak existence obtained through
compactness and identification, the Yamada--Watanabe principle yields a
probabilistically strong solution; uniqueness then removes the subsequence and
gives convergence of the whole sequence.

\paragraph{Organization of the paper.}
In Section~\ref{sec:preliminaries}, we introduce the notation, recall the
entropy and conditional-law estimates from the companion paper, and establish
deterministic estimates for the logarithmic weights and predecessor sums. In
Section~\ref{sec:weighted-tightness}, we prove tightness of the countable
family of logarithmically weighted fluctuation fields. In
Section~\ref{sec:identification-inputs}, we derive the finite-particle
identities and identify the limits of the leading interaction term, the
initial fluctuations, and the martingale noise. In
Section~\ref{sec:limit-identification}, we establish conditional-measure
replacement and use it to identify every subsequential limit. In
Section~\ref{sec:uniqueness-main-result}, we prove uniqueness for the limiting
hierarchy and complete the proof of the main result.

\section{Preliminaries}\label{sec:preliminaries}

In this section, we introduce the notation and give the estimates used
throughout the proof. We first fix the function spaces and notation, then recall the entropy
and conditional-law estimates from \cite{WangZhaoSequentialEntropy}, and
finally prove the deterministic bounds for the logarithmic weights and
predecessor sums.

\subsection{Notation}\label{subsec:notation}

We write \(\mathbb N_0:=\{0,1,2,\ldots\}\), \(\Law(Z)\) for the law of a random variable \(Z\), \(H(\mu\mid\nu)\) for the relative entropy of \(\mu\) with respect to \(\nu\), and \(\|\cdot\|_{\mathrm{TV}}\) for total variation distance. The pairing between a signed measure or distribution \(\gamma\) and a test function \(\varphi\) is denoted by \(\langle\gamma,\varphi\rangle\).
Constants denoted by \(C\) may change from line to line; constants with subscripts indicate the parameters on which they may depend, but never on \(N\) or on the particle label unless this is stated explicitly.

For \(s\in\R\), \(H^s=H^s(\R^d)\) denotes the usual Sobolev space. For
\(\alpha\ge0\), let \(G_\alpha\) denote the positive symmetric convolution
kernel for \((I-\Delta)^{-\alpha}\). Then, for finite signed measures \(\eta\),
\[
\|\eta\|_{H^{-\alpha}}^2
=
\iint_{\R^d\times\R^d}
G_\alpha(x-y)\,\eta(\dd x)\,\eta(\dd y),
\]
whenever the right-hand side is finite.

For \(m\ge0\) and \(\alpha\ge0\), set
\[
\langle x\rangle=(1+|x|^2)^{1/2},
\qquad
\|\phi\|_{H^{\alpha,m}}:=\|\langle\cdot\rangle^m\phi\|_{H^\alpha}.
\]
Let \(H^{\alpha,m}\) be the completion of \(C_c^\infty(\R^d)\) under this norm. The weighted negative Sobolev space is
\[
H^{-\alpha,-m}:=(H^{\alpha,m})'.
\]
Thus \(H^{-\alpha,0}=H^{-\alpha}\). For \(\alpha,m\ge0\), write
\[
E_{\alpha,m}:=C([0,T];H^{-\alpha,-m}).
\]
For the fixed parameters \(\beta>d/2+4\) and \(m>d/2+1\), the path space in
the main result is
\[
E=E_{\beta,m}=C([0,T];H^{-\beta,-m}).
\]
We equip \(E\) with the uniform metric
\[
\rho_E(z,\widetilde z)
:=
\sup_{t\le T}\|z_t-\widetilde z_t\|_{H^{-\beta,-m}}.
\]
On \(E^{\mathbb N_0}\), we use the standard product metric
\[
d_E(\mathbf z,\widetilde{\mathbf z})
:=
\sum_{n=0}^\infty 2^{-(n+1)}
\bigl(1\wedge \rho_E(z^n,\widetilde z^n)\bigr).
\]
The metric \(d_E\) induces the countable product topology. For
\(0\le\alpha_0<\alpha_1\) and \(m>0\), the embedding
\begin{equation}\label{eq:weighted-rellich-embedding}
H^{-\alpha_0}\hookrightarrow H^{-\alpha_1,-m}
\end{equation}
is compact. The Sobolev and Rellich embeddings used here can be found, for
instance, in \cite{triebel1992theoryII}. If \(\alpha>d/2+1\) and \(m\ge0\), Sobolev
embedding gives
\begin{equation}\label{eq:delta-gradient-sobolev-bound}
\sup_{x\in\R^d}
\bigl(
\|\delta_x\|_{H^{-\alpha,-m}}
+
\|\nabla\delta_x\|_{H^{-\alpha,-m}}
\bigr)
<\infty.
\end{equation}

For the full empirical measure we write
\[
\mu_t^N:=\frac1N\sum_{i=1}^N\delta_{X_t^i},
\]
and, for \(i\ge2\),
\[
\mu_t^{i-1}:=\frac1{i-1}\sum_{j<i}\delta_{X_t^j}.
\]
If \(g:(0,1]\to\R\) is a deterministic weight, set
\[
Y_t^N(g):=
\frac1{\sqrt N}
\sum_{i=1}^N
g(i/N)(\delta_{X_t^i}-\bar\rho_t).
\]
We write \(Y^N(g)=(Y_t^N(g))_{t\in[0,T]}\). For
\[
h_n(u):=\frac{\log(1/u)^n}{n!},\qquad u\in(0,1],
\]
we have \(Y^{N,n}=Y^N(h_n)\). Conditional
expectations in the replacement arguments are taken with respect to
\[
\mathscr G_t^{i-1}:=\sigma(X_t^1,\ldots,X_t^{i-1}),
\qquad i\ge2.
\]
For path laws, \(P^{1:i}_{[0,t]}\) denotes the law of \((X^1,\ldots,X^i)\) on \([0,t]\), and \(\bar P_{[0,t]}\) denotes the McKean--Vlasov path law.

\subsection{Entropy and conditional estimates}
\label{subsec:entropy-conditional}

Define the interaction error by
\[
\Delta_t^1:=-(K*\bar\rho_t)(X_t^1),
\qquad
\Delta_t^i
:=
(K*\mu_t^{i-1})(X_t^i)
-
(K*\bar\rho_t)(X_t^i),
\qquad i\ge2.
\]

For \(i\ge2\), set
\[
R_i(t)
:=
H\!\left(
P^{1:i}_{[0,t]}
\,\middle|\,
P^{1:i-1}_{[0,t]}\otimes \bar P_{[0,t]}
\right).
\]

The following lemma gives the estimates from
\cite{WangZhaoSequentialEntropy} used below in the smooth convolutional setting
with identity diffusion; the proof gives the precise references.

\begin{lemma}\label{lem:paper1-estimates}
Under \hyperref[ass:A]{Assumption~A}, the following statements hold.
\begin{enumerate}
\item[\textup{(i)}]
There is a constant \(C_T<\infty\), depending only on \(T,d\) and the coefficient bounds in \hyperref[ass:A]{Assumption~A}, but not on \(i\) or \(N\), such that for \(i\ge2\),
\[
R_i(T)\le C_T(i-1)^{-1},
\qquad
\mathbb E\int_0^T|\Delta_t^i|^2\,\dd t\le C_T(i-1)^{-1}.
\]

\item[\textup{(ii)}]
Write \(X_t^{1:i}:=(X_t^1,\ldots,X_t^i)\) and set
\(\nu_t^1:=\Law(X_t^1)\). For every \(i\ge2\), there is a probability kernel
\[
f_t^i:(\R^d)^{i-1}\to\mathcal P(\R^d),
\qquad
(t,x^{1:i-1})\longmapsto f_t^i(x^{1:i-1},\cdot),
\]
jointly Borel measurable in \((t,x^{1:i-1})\), such that
\[
\dd t\otimes\Law(X_t^{1:i})(\dd x^{1:i})
=
\dd t\otimes\Law(X_t^{1:i-1})(\dd x^{1:i-1})
\,f_t^i(x^{1:i-1},\dd x_i).
\]
Define
\[
\nu_t^i(\omega)
:=
f_t^i(X_t^{1:i-1}(\omega),\cdot).
\]
Then \((t,\omega)\mapsto\nu_t^i(\omega)\) is jointly measurable and
\[
\nu_t^i=\Law(X_t^i\mid\mathscr G_t^{i-1})
\quad\text{in the }\dd t\otimes\mathbb P\text{-a.s.\ sense}.
\]
For \(i\ge2\), set
\[
\bar\nu_t^{i-1}:=\frac1{i-1}
\sum_{j=1}^{i-1}\nu_t^j.
\]

\item[\textup{(iii)}]
For Lebesgue-a.e. \(t\in[0,T]\) and \(i\ge2\),
\[
\mathbb E H(\nu_t^i\mid\bar\rho_t)
\le R_i(t)
\le R_i(T)
\le C_T(i-1)^{-1},
\]
\[
\mathbb E\|\nu_t^i-\bar\rho_t\|_{\mathrm{TV}}^2
\le C_T(i-1)^{-1},
\]
and
\[
\mathbb E\|\bar\nu_t^{i-1}-\bar\rho_t\|_{\mathrm{TV}}^2
\le C_T(i-1)^{-1}.
\]

\item[\textup{(iv)}]
For every deterministic bounded measurable function \(\phi_t(x,y)\), uniformly
for a.e.\ \(t\in[0,T]\), \(i\ge2\), and \(\lambda\in\R\),
\[
\begin{aligned}
&\mathbb E
\int
\exp\!\left[
\frac{\lambda}{i-1}
\left(
\phi_t(x,X_t^1)-\int \phi_t(x,y)\Law(X_t^1)(\dd y)
\right.\right.\\
&\hspace{4.7cm}\left.\left.
+\sum_{j=2}^{i-1}
\left[
\phi_t(x,X_t^j)-\int \phi_t(x,y)\nu_t^j(\dd y)
\right]
\right)
\right]
\bar\rho_t(\dd x)\\
&\qquad\le
\exp\left(C_\phi\lambda^2(i-1)^{-1}\right),
\end{aligned}
\]
where \(C_\phi<\infty\) depends only on a uniform bound for \(\phi\).
\end{enumerate}
\end{lemma}

\begin{proof}
The incremental entropy estimate in \textup{(i)} is the i.i.d.\ specialization
of \cite[Theorem~1]{WangZhaoSequentialEntropy}; its Girsanov identity is
\cite[Lemma~3.1]{WangZhaoSequentialEntropy}. These results give both
\(R_i(T)\lesssim(i-1)^{-1}\) and the corresponding estimate for
\(\Delta^i\). The kernels, disintegration identity, and
\(\dd t\otimes\mathbb P\)-a.s.\ conditional-law interpretation in
\textup{(ii)} are precisely
\cite[Lemma~3.2]{WangZhaoSequentialEntropy}. By data processing from path
space to the time marginal,
\(\mathbb EH(\nu_t^i\mid\bar\rho_t)\le R_i(t)\).
Applying Pinsker's inequality then proves the first total-variation estimate
in \textup{(iii)}.

For the averaged estimate in \textup{(iii)}, we apply Minkowski's inequality.
Using \(\|\Law(X_t^1)-\bar\rho_t\|_{\mathrm{TV}}\le2\) together with the
preceding individual estimates, we obtain
\[
\begin{aligned}
\left(
\mathbb E\|\bar\nu_t^{i-1}-\bar\rho_t\|_{\mathrm{TV}}^2
\right)^{1/2}
&\le
\frac1{i-1}
\left(
2+
\sum_{j=2}^{i-1}
\left(
\mathbb E\|\nu_t^j-\bar\rho_t\|_{\mathrm{TV}}^2
\right)^{1/2}
\right)\\
&\le \frac{C_T}{\sqrt{i-1}}.
\end{aligned}
\]

For \textup{(iv)}, fix \(t\), \(i\), and \(x\) as in the statement, and set
\[
M_\phi:=\sup_{s\le T}\sup_{x,y\in\mathbb R^d}|\phi_s(x,y)|,
\qquad
D_j(x):=\phi_t(x,X_t^j)-\int_{\mathbb R^d}\phi_t(x,y)\nu_t^j(\dd y),
\quad 1\le j<i,
\]
where \(\mathscr G_t^0\) is the trivial sigma-field. By the definition of
\(\nu_t^j\), the random variable \(D_j(x)\) is
\(\mathscr G_t^j\)-measurable and
\[
\mathbb E\bigl[D_j(x)\mid\mathscr G_t^{j-1}\bigr]=0.
\]
Conditionally on \(\mathscr G_t^{j-1}\), its range is contained in an
interval of length at most \(2M_\phi\). The conditional Hoeffding lemma
therefore gives, for every \(\theta\in\mathbb R\),
\[
\mathbb E\bigl[\exp(\theta D_j(x))\mid\mathscr G_t^{j-1}\bigr]
\le \exp\!\left(\frac{M_\phi^2\theta^2}{2}\right).
\]
Iterating this estimate over \(j=1,\ldots,i-1\), as in the proof of
\cite[Lemma~3.3]{WangZhaoSequentialEntropy}, yields
\[
\mathbb E\exp\!\left(\theta\sum_{j=1}^{i-1}D_j(x)\right)
\le \exp\!\left(\frac{(i-1)M_\phi^2\theta^2}{2}\right).
\]
Taking \(\theta=\lambda/(i-1)\) and integrating with respect to
\(\bar\rho_t(\dd x)\) proves \textup{(iv)} with
\(C_\phi=M_\phi^2/2\).
\end{proof}

\subsection{Logarithmic weights and predecessor sums}

The operator \(\mathsf T\) is the limit of predecessor summation. Starting from
\(h_0=1\), define
\[
h_0(u)=1,\qquad h_{n+1}=\mathsf T h_n,
\qquad
(\mathsf T f)(u)=\int_u^1\frac{f(v)}{v}\,\dd v.
\]

\begin{lemma}\label{lem:logarithmic-weights}
For every \(n\ge0\), with \(h_0\equiv1\),
\[
h_n(u)=\frac{\log(1/u)^n}{n!}.
\]
Moreover, for all \(n,k\ge0\),
\[
\int_0^1 h_n(u)h_k(u)\,\dd u
=
\frac{(n+k)!}{n!\,k!}.
\]
\end{lemma}

\begin{proof}
The formula for \(h_n\) follows by induction from the change of variables \(v=e^{-s}\):
\[
(\mathsf T h_n)(e^{-s})=\int_0^s \frac{r^n}{n!}\,\dd r=\frac{s^{n+1}}{(n+1)!}.
\]
For the covariance factor, set \(u=e^{-s}\). Then
\[
\int_0^1 h_n(u)h_k(u)\,\dd u
=
\frac1{n!\,k!}\int_0^\infty s^{n+k}e^{-s}\,\dd s
=
\frac{(n+k)!}{n!\,k!}.
\]
\end{proof}

The following elementary bounds on \(h_n\) will be used in the tightness and
replacement estimates.

\begin{lemma}\label{lem:elementary-weights}
For every fixed \(\ell\ge0\),
\[
\max_{0\le n\le \ell+1}
\left(
\|h_n\|_{L^2(0,1)}
+
\int_0^1 h_n(u)u^{-1/2}\,\dd u
\right)<\infty,
\]
and
\[
\frac1{\sqrt N}
\max_{0\le n\le \ell+1}
\max_{1\le i\le N} h_n(i/N)\to0.
\]
\end{lemma}

\begin{proof}
By Lemma~\ref{lem:logarithmic-weights}, \(h_n(u)=\log(1/u)^n/n!\). With \(u=e^{-s}\),
\[
\int_0^1 h_n(u)^2\,\dd u
=
\frac1{(n!)^2}\int_0^\infty s^{2n}e^{-s}\,\dd s
<\infty
\]
and
\[
\int_0^1 h_n(u)u^{-1/2}\,\dd u
=
\frac1{n!}\int_0^\infty s^ne^{-s/2}\,\dd s
<\infty.
\]
The maximum is over finitely many values of \(n\). Also
\[
\max_{1\le i\le N}h_n(i/N)\le C_{\ell}(1+\log N)^{\ell+1},
\]
so the \(N^{-1/2}\) maximum tends to zero.
\end{proof}

We also need the following logarithmic sums, which vanish at rates compatible
with the entropy and conditional-bias bounds in
Lemma~\ref{lem:paper1-estimates}\textup{(iii)}.

\begin{lemma}\label{lem:polylog-summability}
For every fixed \(\ell\ge0\),
\[
\max_{0\le n\le \ell+1}
\frac1N\sum_{i=2}^N
\frac{h_n(i/N)^2}{i-1}\to0,
\]
and
\[
\max_{0\le n\le \ell+1}
\frac1{\sqrt N}\sum_{i=2}^N
\frac{h_n(i/N)}{i-1}\to0.
\]
\end{lemma}

\begin{proof}
For \(0\le n\le \ell+1\),
\[
h_n(i/N)\le C_{\ell}(1+\log(N/i))^{\ell+1}
\le C_{\ell}(1+\log N)^{\ell+1}.
\]
Therefore
\[
\frac1N\sum_{i=2}^N
\frac{h_n(i/N)^2}{i-1}
\le
\frac{C_{\ell}(1+\log N)^{2\ell+2}}N
\sum_{i=2}^N\frac1{i-1}
\longrightarrow0,
\]
and
\[
\frac1{\sqrt N}\sum_{i=2}^N
\frac{h_n(i/N)}{i-1}
\le
\frac{C_{\ell}(1+\log N)^{\ell+1}}{\sqrt N}
\sum_{i=2}^N\frac1{i-1}\to0,
\]
uniformly over the indicated values of \(n\).
\end{proof}

The quantity \(\mathcal W_n\) below contains the weight terms that appear in
the energy estimate of Proposition~\ref{prop:weighted-apriori}. Its growth
follows directly from the explicit formula for \(h_n\).

\begin{lemma}\label{lem:logarithmic-energy-weight-growth}
For \(n\ge0\), set
\[
\mathcal W_n
:=
\left[
\|h_n\|_{L^2(0,1)}
+
\int_0^1 h_n(u)u^{-1/2}\,\dd u
\right]^2
+
\sup_{N\ge1}\frac{|h_n(1/N)|^2}{N}.
\]
There exists \(C<\infty\), independent of \(n\), such that
\[
\mathcal W_n\le C4^n,\qquad n\ge0.
\]
\end{lemma}

\begin{proof}
By Lemma~\ref{lem:logarithmic-weights} and the change of variables \(u=e^{-s}\),
\[
\|h_n\|_{L^2(0,1)}^2
=
\frac{(2n)!}{(n!)^2}
\le 4^n
\]
and
\[
\int_0^1 h_n(u)u^{-1/2}\,\dd u
=
\frac1{n!}\int_0^\infty s^n e^{-s/2}\,\dd s
=
2^{n+1}.
\]
Thus the square of the sum of these two terms is bounded by \(C4^n\).
For the boundary term, if \(n=0\) then \(h_0(1/N)^2/N\le1\). If \(n\ge1\), set
\(x=\log N\). Then
\[
\frac{|h_n(1/N)|^2}{N}
=
e^{-x}\frac{x^{2n}}{(n!)^2}
\le
\sup_{y\ge0}e^{-y}\frac{y^{2n}}{(n!)^2}
=
e^{-2n}\frac{(2n)^{2n}}{(n!)^2}
\le 4^n,
\]
where the last inequality uses \(n!\ge(n/e)^n\). Combining the three estimates
proves \(\mathcal W_n\le C4^n\), after increasing \(C\) to include \(n=0\).
\end{proof}

For finite \(N\), define the discrete predecessor operator by
\begin{equation}\label{eq:TN-def}
(\mathsf T_N f)(j/N)
=
\sum_{i=j+1}^N\frac{f(i/N)}{i-1}.
\end{equation}
In the sequel \(\mathsf T_N f\) appears with \(f=h_n\), and we set
\((\mathsf T_N h_n)(1)=0\).

We will use the following Riemann limit for the martingale covariance.

\begin{lemma}\label{lem:riemann-weights}
For every fixed \(n,k\ge0\),
\[
\frac1N\sum_{i=1}^N h_n(i/N)h_k(i/N)
\to
\int_0^1 h_n(u)h_k(u)\,\dd u
=
\frac{(n+k)!}{n!\,k!}.
\]
\end{lemma}

\begin{proof}
Write \(u=e^{-s}\). Then \(h_n(u)h_k(u)=s^{n+k}/(n!\,k!)\), which is integrable against \(e^{-s}\,\dd s\). Split \((0,1]\) into \((0,\varepsilon]\cup[\varepsilon,1]\). Since \(h_nh_k\) is nonincreasing, its right-endpoint sum on \((0,\varepsilon]\) is bounded by the corresponding integral and is therefore uniformly small as \(\varepsilon\downarrow0\). On \([\varepsilon,1]\), ordinary Riemann-sum convergence applies. The value of the integral is given by Lemma~\ref{lem:logarithmic-weights}.
\end{proof}

The next lemma controls \(\mathsf T_N h_n-h_{n+1}\) in \(L^2\) and in a
predecessor-weighted \(L^1\) norm.

\begin{lemma}\label{lem:TN-error-summability}
For every fixed \(n\ge0\), set
\[
e_N^{(n)}(j):=(\mathsf T_N h_n)(j/N)-h_{n+1}(j/N),
\qquad 1\le j\le N.
\]
Then
\[
\frac1N\sum_{j=1}^N |e_N^{(n)}(j)|^2\to0,
\]
and
\[
\frac1{\sqrt N}\sum_{j=1}^N
\frac{|e_N^{(n)}(j)|}{\sqrt j}\to0.
\]
\end{lemma}

\begin{proof}
Set \(u_j=j/N\). Lemma~\ref{lem:logarithmic-weights} and the definition of
\(\mathsf T_N\) in \eqref{eq:TN-def}
give the pointwise bound
\begin{equation}\label{eq:TN-error-pointwise-bound}
|e_N^{(n)}(j)|
\le
C_n(1+\log(N/j))^{n+1}.
\end{equation}
Indeed, for \(i\ge2\),
\[
\frac{h_n(i/N)}{i-1}
\le
2\frac{h_n(i/N)}i
\le
2\int_{i-1}^i
\frac{h_n(x/N)}x\,\dd x,
\]
because \(x\mapsto x^{-1}h_n(x/N)\) is nonincreasing. Summing from
\(i=j+1\) to \(N\) gives
\[
0\le (\mathsf T_N h_n)(j/N)
\le
2\int_j^N\frac{h_n(x/N)}x\,\dd x
=2h_{n+1}(j/N).
\]
Together with
\(|e_N^{(n)}(j)|\le (\mathsf T_N h_n)(j/N)+h_{n+1}(j/N)\), this proves
\eqref{eq:TN-error-pointwise-bound}.
This bound is square-summable at the logarithmic scale. Hence the contribution
of \(j\le\varepsilon N\) to \(N^{-1}\sum_j |e_N^{(n)}(j)|^2\) is uniformly
small as \(\varepsilon\downarrow0\). On \(j>\varepsilon N\), write
\[
(\mathsf T_N h_n)(j/N)
=
\sum_{i=j+1}^N\frac{h_n(i/N)}i
+r_{N,j},
\qquad
r_{N,j}
:=
\sum_{i=j+1}^N\frac{h_n(i/N)}{i(i-1)}.
\]
The first term is the right-endpoint Riemann sum for
\[
\int_{j/N}^1\frac{h_n(v)}v\,\dd v
=
h_{n+1}(j/N).
\]
The integrand is smooth and bounded on \([\varepsilon,1]\), so this convergence
is uniform for \(j>\varepsilon N\). Moreover,
\[
\sup_{j>\varepsilon N}|r_{N,j}|
\le
C_{n,\varepsilon}
\sum_{i>\varepsilon N}\frac1{i(i-1)}
\le \frac{C_{n,\varepsilon}}N.
\]
Combining the small-\(j\) and large-\(j\) ranges proves the first convergence.

For the predecessor-weighted \(L^1\) convergence, use
\eqref{eq:TN-error-pointwise-bound}. For \(j\le\varepsilon N\), by
\eqref{eq:TN-error-pointwise-bound},
\[
\frac1{\sqrt N}\sum_{j\le\varepsilon N}\frac{|e_N^{(n)}(j)|}{\sqrt j}
\le
C_n\int_0^\varepsilon u^{-1/2}(1+\log(1/u))^{n+1}\,\dd u+o(1),
\]
which is small as \(\varepsilon\downarrow0\). For \(j>\varepsilon N\), the
uniform convergence of \((\mathsf T_N h_n)(j/N)\) to
\(h_{n+1}(j/N)\) on
\([\varepsilon,1]\) gives
\[
\sup_{j>\varepsilon N}|e_N^{(n)}(j)|\to0.
\]
Hence
\[
\frac1{\sqrt N}\sum_{j>\varepsilon N}\frac{|e_N^{(n)}(j)|}{\sqrt j}
\le
\sup_{j>\varepsilon N}|e_N^{(n)}(j)|
\frac1{\sqrt N}\sum_{j>\varepsilon N}j^{-1/2}
\to0.
\]
Combining the two ranges proves the claim.
\end{proof}

The product weights in the quadratic variations satisfy the following bounds.

\begin{lemma}\label{lem:product-weight-bounds}
For every fixed \(n,k\ge0\),
\[
\frac1N\sum_{i=1}^N |h_n(i/N)h_k(i/N)|^2\le C_{n,k},
\]
and
\[
\frac1N
\left(
|h_n(1/N)h_k(1/N)|
+
\sum_{i=2}^N
\frac{|h_n(i/N)h_k(i/N)|}{\sqrt{i-1}}
\right)
\to0.
\]
\end{lemma}

\begin{proof}
The function \(u\mapsto |h_n(u)h_k(u)|^2\) is nonincreasing. Hence, by the same
right-endpoint argument as in the proof of
Lemma~\ref{lem:riemann-weights},
\[
\frac1N\sum_{i=1}^N |h_n(i/N)h_k(i/N)|^2
\le
\int_0^1 |h_n(u)h_k(u)|^2\,\dd u
=
\frac{(2n+2k)!}{(n!)^2(k!)^2}.
\]
For the second bound, use
\(
|h_n(u)h_k(u)|
\le C_{n,k}(1+\log(1/u))^{n+k}
\)
to obtain
\[
\frac{|h_n(1/N)h_k(1/N)|}{N}
\le
\frac{C_{n,k}(1+\log N)^{n+k}}{N}\to0,
\]
and
\[
\frac1N\sum_{i=2}^N
\frac{|h_n(i/N)h_k(i/N)|}{\sqrt{i-1}}
\le
C_{n,k}(1+\log N)^{n+k}N^{-1/2}\to0.
\]
\end{proof}

\section{Tightness of the countable logarithmic hierarchy}\label{sec:weighted-tightness}

In this section, we prove tightness of the countable logarithmic hierarchy. We
first establish a uniform negative-Sobolev energy bound for each \(Y^{N,n}\),
including the sequential interaction error. We then combine fixed-time
tightness with time equicontinuity to obtain tightness of each coordinate and
of the full countable family.

Throughout this section, \(\beta_\ast\) satisfies
\[
\frac d2+2<\beta_\ast<\beta-2.
\]
The energy estimates are carried out in \(H^{-\beta_\ast}\), while tightness
is proved in \(E=C([0,T];H^{-\beta,-m})\).

The notation \(Y^{N,n}=Y^N(h_n)\), \(G_{\beta_\ast}\), and \(\Delta^i\) is
as in Sections~\ref{subsec:notation} and
\ref{subsec:entropy-conditional}. For \(n\ge0\), set
\[
\Phi_t^{N,n}:=G_{\beta_\ast}*Y_t^{N,n}.
\]
The particle equations become
\[
\dd X_t^i=
\bigl(b(t,X_t^i)+(K*\bar\rho_t)(X_t^i)+\Delta_t^i\bigr)\,\dd t+\dd B_t^i,
\qquad i\ge1.
\]

\begin{proposition}\label{prop:weighted-energy-identity}
For every \(N\ge1\) and \(n\ge0\), the following identity holds almost surely
for every \(t\in[0,T]\):
\begin{align}
\|Y_t^{N,n}\|_{H^{-\beta_\ast}}^2
&=
\|Y_0^{N,n}\|_{H^{-\beta_\ast}}^2
+
c_{\beta_\ast}\left(\frac1N\sum_{i=1}^N h_n(i/N)^2\right)t\notag\\
&\quad+
2\int_0^t
\left\langle Y_s^{N,n},
\frac12\Delta\Phi_s^{N,n}+
\bigl(b(s,\cdot)+K*\bar\rho_s\bigr)\cdot\nabla\Phi_s^{N,n}
\right\rangle\,\dd s\notag\\
&\quad+
\frac2{\sqrt N}\sum_{i=1}^N h_n(i/N)
\int_0^t
\nabla\Phi_s^{N,n}(X_s^i)\cdot\Delta_s^i\,\dd s\notag\\
&\quad+
\mathfrak M_t^{N,n},
\label{eq:weighted-energy-identity}
\end{align}
where \(c_{\beta_\ast}:=-\Delta G_{\beta_\ast}(0)\ge0\), and
\begin{equation*}
\mathfrak M_t^{N,n}
=
\frac2{\sqrt N}
\sum_{i=1}^N h_n(i/N)
\int_0^t
\nabla\Phi_s^{N,n}(X_s^i)\cdot\dd B_s^i.
\end{equation*}
The process \(\mathfrak M^{N,n}\) is a real-valued continuous square-integrable
martingale.
\end{proposition}

\begin{proof}
This is the weighted version of
\cite[Lemma~4.1]{WangZhaoSequentialEntropy}.  Using the weak Fokker--Planck
equation for \(\bar\rho\) together with It\^o's formula for the particles, we
obtain, for \(\varphi\in C_c^\infty(\R^d)\),
\begin{align}
\dd\langle Y_t^{N,n},\varphi\rangle
&=
\left\langle Y_t^{N,n},
\frac12\Delta\varphi+\bigl(b(t,\cdot)+K*\bar\rho_t\bigr)\cdot\nabla\varphi
\right\rangle\,\dd t\notag\\
&\quad+
\frac1{\sqrt N}\sum_{i=1}^N h_n(i/N)
\nabla\varphi(X_t^i)\cdot\Delta_t^i\,\dd t
+
\dd M_t^{N,n}(\varphi),
\label{eq:weak-YN}
\end{align}
where
\[
M_t^{N,n}(\varphi)=
\frac1{\sqrt N}\sum_{i=1}^N h_n(i/N)
\int_0^t\nabla\varphi(X_s^i)\cdot \dd B_s^i.
\]
Let \(L_t^*\) denote the distributional adjoint of \(L_t\), and set
\[
\mathcal A_t^{N,n}
:=
-\frac1{\sqrt N}\sum_{i=1}^N h_n(i/N)
\nabla\!\cdot\bigl(\Delta_t^i\delta_{X_t^i}\bigr),
\]
\[
\dd\mathcal N_t^{N,n}
:=
-\frac1{\sqrt N}\sum_{i=1}^N h_n(i/N)
\nabla\!\cdot\bigl(\delta_{X_t^i}\,\dd B_t^i\bigr).
\]
Since \(\beta_\ast>d/2+2\), the map \(x\mapsto\delta_x\) has bounded
\(H^{-\beta_\ast}\)-valued derivatives through order two.  Assumption~A also
gives bounded coefficients and \(|\Delta_t^i|\le2\|K\|_\infty\).  Thus the
drift terms are \(H^{-\beta_\ast}\)-integrable and the stochastic integral is
well defined in \(H^{-\beta_\ast}\).  Equation~\eqref{eq:weak-YN} is therefore
equivalent to the following \(H^{-\beta_\ast}\)-valued identity:
\[
\dd Y_t^{N,n}
=L_t^*Y_t^{N,n}\,\dd t
+\mathcal A_t^{N,n}\,\dd t
+\dd\mathcal N_t^{N,n}.
\]
For \(u,v\in H^{-\beta_\ast}\),
\[
(u,v)_{H^{-\beta_\ast}}
=
\langle v,G_{\beta_\ast}*u\rangle.
\]
Applying the Hilbert-space Itô formula
\cite[Theorem~6.1.1]{rockner2015spde} in \(H^{-\beta_\ast}\) to
\(u\mapsto\|u\|_{H^{-\beta_\ast}}^2\), and using
\(\Phi_t^{N,n}=G_{\beta_\ast}*Y_t^{N,n}\), gives
\[
\begin{aligned}
\dd\|Y_t^{N,n}\|_{H^{-\beta_\ast}}^2
={}&
2\bigl\langle Y_t^{N,n},L_t\Phi_t^{N,n}\bigr\rangle\,\dd t
+2\bigl\langle\mathcal A_t^{N,n},\Phi_t^{N,n}\bigr\rangle\,\dd t\\
&+\dd[\mathcal N^{N,n}]_t^{H^{-\beta_\ast}}
+2\bigl(Y_t^{N,n},\dd\mathcal N_t^{N,n}\bigr)_{H^{-\beta_\ast}}.
\end{aligned}
\]
For every \(x\in\R^d\) and \(1\le a\le d\), the definition of the inner
product and translation invariance give
\[
\|\partial_a\delta_x\|_{H^{-\beta_\ast}}^2
=
-\partial_{aa}G_{\beta_\ast}(0).
\]
The independence of the Brownian motions therefore gives
\[
\frac{\dd}{\dd t}[\mathcal N^{N,n}]_t^{H^{-\beta_\ast}}
=
\frac1N\sum_{i=1}^N h_n(i/N)^2
\sum_{a=1}^d\|\partial_a\delta_{X_t^i}\|_{H^{-\beta_\ast}}^2
=
c_{\beta_\ast}\frac1N\sum_{i=1}^N h_n(i/N)^2.
\]
Moreover,
\[
\begin{aligned}
\bigl\langle\mathcal A_t^{N,n},\Phi_t^{N,n}\bigr\rangle
&=
\frac1{\sqrt N}\sum_{i=1}^N h_n(i/N)
\nabla\Phi_t^{N,n}(X_t^i)\cdot\Delta_t^i,\\
2\bigl(Y_t^{N,n},\dd\mathcal N_t^{N,n}\bigr)_{H^{-\beta_\ast}}
&=
2\bigl\langle\dd\mathcal N_t^{N,n},\Phi_t^{N,n}\bigr\rangle\\
&=
\frac2{\sqrt N}\sum_{i=1}^N h_n(i/N)
\nabla\Phi_t^{N,n}(X_t^i)\cdot\dd B_t^i
=\dd\mathfrak M_t^{N,n}.
\end{aligned}
\]
For fixed \(N,n\),
\(\sup_{t\le T}\|Y_t^{N,n}\|_{\mathrm{TV}}
\le 2N^{-1/2}\sum_i|h_n(i/N)|\), and \(\nabla G_{\beta_\ast}\) is bounded.
Hence \(\mathfrak M^{N,n}\) is square-integrable.  Integrating in time proves
\eqref{eq:weighted-energy-identity}.
\end{proof}

We next bound the initial fluctuation and the sequential interaction error in
the energy estimate.

\begin{lemma}\label{lem:weighted-J}
For \(n\ge0\), define
\[
J_T^{N,n}
:=
\int_0^T
\left(
\frac1{\sqrt N}\sum_{i=1}^N |h_n(i/N)|\,|\Delta_t^i|
\right)^2\,\dd t.
\]
With \(\mathcal W_n\) defined in
Lemma~\ref{lem:logarithmic-energy-weight-growth}, there is a constant
\(C<\infty\), independent of \(n\) and \(N\), such that
\[
\sup_N
\left(
\mathbb E J_T^{N,n}
+
\mathbb E\|Y_0^{N,n}\|_{H^{-\beta_\ast}}^2
\right)
\le
C\mathcal W_n.
\]
\end{lemma}

\begin{proof}
This is the weighted version of
\cite[Lemmas~4.4 and~4.5]{WangZhaoSequentialEntropy}.
Set
\[
a_i:=h_n(i/N),
\qquad
\xi_i:=\delta_{X_0^i}-\bar\rho_0.
\]
By Assumption~A, the \(\xi_i\) are independent centered
\(H^{-\beta_\ast}\)-valued random variables. Hence the Bessel-kernel
representation gives
\[
\begin{aligned}
\mathbb E\|Y_0^{N,n}\|_{H^{-\beta_\ast}}^2
&=
\frac1N\sum_{i=1}^N a_i^2
\left[
G_{\beta_\ast}(0)
-
\iint G_{\beta_\ast}(x-y)\bar\rho_0(\dd x)\bar\rho_0(\dd y)
\right]\\
&\le
G_{\beta_\ast}(0)\frac1N\sum_{i=1}^N h_n(i/N)^2
\le
G_{\beta_\ast}(0)\|h_n\|_{L^2(0,1)}^2.
\end{aligned}
\]
The last inequality follows because \(h_n^2\) is nonincreasing.

Applying Minkowski's inequality and using
\(|\Delta_t^1|\le\|K\|_\infty\) together with
Lemma~\ref{lem:paper1-estimates}\textup{(i)}, we obtain
\[
\bigl(\mathbb E J_T^{N,n}\bigr)^{1/2}
\le
C_T\left[
\frac{h_n(1/N)}{\sqrt N}
+
\frac1{\sqrt N}\sum_{i=2}^N
\frac{h_n(i/N)}{\sqrt{i-1}}
\right].
\]
For \(i\ge2\), monotonicity of \(h_n\) and
\[
\frac1{\sqrt N\sqrt{i-1}}
\le
C\int_{(i-1)/N}^{i/N}u^{-1/2}\,\dd u
\]
imply
\[
\frac{h_n(i/N)}{\sqrt N\sqrt{i-1}}
\le
C\int_{(i-1)/N}^{i/N}h_n(u)u^{-1/2}\,\dd u.
\]
Summing in \(i\) yields
\[
\begin{aligned}
\bigl(\mathbb E J_T^{N,n}\bigr)^{1/2}
&\le
C_T\left[
\frac{h_n(1/N)}{\sqrt N}
+
\int_0^1 h_n(u)u^{-1/2}\,\dd u
\right]\\
&\le C_T\mathcal W_n^{1/2}.
\end{aligned}
\]
Combining the two estimates proves the lemma.
\end{proof}

\begin{proposition}\label{prop:weighted-apriori}
There exists \(C<\infty\), independent of \(n\), such that for every \(n\ge0\),
\begin{equation}\label{eq:weighted-apriori-bound}
\sup_N
\mathbb E\left[
\sup_{t\le T}
\|Y_t^{N,n}\|_{H^{-\beta_\ast}}^2
+
\int_0^T
\|Y_t^{N,n}\|_{H^{-\beta_\ast+1}}^2\,\dd t
\right]
\le C4^n.
\end{equation}
\end{proposition}

\begin{proof}
This is the weighted version of
\cite[Lemmas~4.2, 4.3, and~4.6]{WangZhaoSequentialEntropy}.
Fix \(n\ge0\). All constants below are independent of \(N\) and \(n\). Write
\[
A_{N,n}:=\frac1N\sum_{i=1}^N h_n(i/N)^2.
\]
For \(0\le t\le T\), set
\[
\begin{aligned}
Z_t^{N,n}&:=\sup_{r\le t}\|Y_r^{N,n}\|_{H^{-\beta_\ast}}^2,
&
D_t^{N,n}&:=\int_0^t
\|Y_s^{N,n}\|_{H^{-\beta_\ast+1}}^2\,\dd s,\\
J_t^{N,n}&:=\int_0^t
\left(
\frac1{\sqrt N}\sum_{i=1}^N |h_n(i/N)|\,|\Delta_s^i|
\right)^2\,\dd s,
&
\mathfrak M_t^{N,n,*}&:=\sup_{r\le t}|\mathfrak M_r^{N,n}|.
\end{aligned}
\]
Since \(h_n^2\) is nonincreasing,
\[
A_{N,n}\le\int_0^1h_n(u)^2\,\dd u\le\mathcal W_n.
\]

{
Write \(V_s=b(s,\cdot)+K*\bar\rho_s\). Since \(b,K\in C_b^\infty\), we have
\(V_s\in C_b^\infty(\R^d;\R^d)\), and all its spatial derivatives are bounded
uniformly in \(s\in[0,T]\). The Sobolev multiplier estimate therefore gives,
for every \(r\ge0\),
\begin{equation}\label{eq:uniform-V-multiplier-bound}
\sup_{s\le T}\|V_s f\|_{H^\sigma}
\le C_r\|f\|_{H^\sigma},
\qquad |\sigma|\le r.
\end{equation}
}
The diffusion term satisfies
\[
2\left\langle \eta,\frac12\Delta(G_{\beta_\ast}*\eta)\right\rangle
=
-\|\eta\|_{H^{-\beta_\ast+1}}^2
+\|\eta\|_{H^{-\beta_\ast}}^2.
\]
{By \eqref{eq:uniform-V-multiplier-bound} and Bessel-potential
duality,}
\[
\|V_s\cdot\nabla(G_{\beta_\ast}*\eta)\|_{H^{\beta_\ast}}
\le C\|\eta\|_{H^{-\beta_\ast+1}}.
\]
Hence, by duality and Young's inequality,
\[
2\left|\left\langle \eta,
V_s\cdot\nabla(G_{\beta_\ast}*\eta)
\right\rangle\right|
\le
\frac14\|\eta\|_{H^{-\beta_\ast+1}}^2
+C\|\eta\|_{H^{-\beta_\ast}}^2.
\]
Since \(\beta_\ast>d/2+2\), Sobolev embedding also yields
\[
\|\nabla\Phi_s^{N,n}\|_\infty
\le C\|Y_s^{N,n}\|_{H^{-\beta_\ast+1}}.
\]
Consequently,
\[
\begin{aligned}
&\left|
\frac2{\sqrt N}\sum_{i=1}^N h_n(i/N)
\nabla\Phi_s^{N,n}(X_s^i)\cdot\Delta_s^i
\right|\\
&\qquad\le
\frac14\|Y_s^{N,n}\|_{H^{-\beta_\ast+1}}^2
+C\left(
\frac1{\sqrt N}\sum_{i=1}^N |h_n(i/N)|\,|\Delta_s^i|
\right)^2.
\end{aligned}
\]
Combining these estimates with \eqref{eq:weighted-energy-identity} gives,
almost surely for every \(t\le T\),
\[
\begin{aligned}
\|Y_t^{N,n}\|_{H^{-\beta_\ast}}^2
+\frac12D_t^{N,n}
\le{}&
\|Y_0^{N,n}\|_{H^{-\beta_\ast}}^2
+C A_{N,n}t
+C\int_0^t\|Y_s^{N,n}\|_{H^{-\beta_\ast}}^2\,\dd s\\
&+C J_t^{N,n}+\mathfrak M_t^{N,n}.
\end{aligned}
\]

Dropping the dissipation term, taking the supremum over \(r\le t\), and using
\(J_r^{N,n}\le J_t^{N,n}\), we obtain
\[
Z_t^{N,n}
\le
\|Y_0^{N,n}\|_{H^{-\beta_\ast}}^2
+C_T A_{N,n}
+CJ_t^{N,n}
+\mathfrak M_t^{N,n,*}
+C\int_0^tZ_s^{N,n}\,\dd s.
\]
Gronwall's inequality therefore gives
\[
Z_T^{N,n}
\le
C_T\left(
\|Y_0^{N,n}\|_{H^{-\beta_\ast}}^2
+A_{N,n}+J_T^{N,n}+\mathfrak M_T^{N,n,*}
\right).
\]
Returning to the previous energy inequality at \(t=T\), and using
\(\int_0^T\|Y_s^{N,n}\|_{H^{-\beta_\ast}}^2\,\dd s\le TZ_T^{N,n}\),
we obtain
\[
Z_T^{N,n}+D_T^{N,n}
\le
C_T\left(
\|Y_0^{N,n}\|_{H^{-\beta_\ast}}^2
+A_{N,n}+J_T^{N,n}+\mathfrak M_T^{N,n,*}
\right).
\]

The quadratic variation of \(\mathfrak M^{N,n}\) satisfies
\[
\langle\mathfrak M^{N,n}\rangle_T
\le
C A_{N,n}D_T^{N,n}.
\]
Thus the BDG inequality \cite{da2014stochastic}, followed by Young's
inequality, gives, for every \(\varepsilon>0\),
\[
\mathbb E\mathfrak M_T^{N,n,*}
\le
C A_{N,n}^{1/2}\mathbb E\bigl[(D_T^{N,n})^{1/2}\bigr]
\le
\varepsilon\mathbb E D_T^{N,n}+C_\varepsilon A_{N,n}.
\]
Taking expectations in the previous inequality and choosing \(\varepsilon\)
small enough to absorb the dissipation term, Lemma~\ref{lem:weighted-J} and
the bound \(A_{N,n}\le\mathcal W_n\) yield
\[
\sup_N\mathbb E\left[Z_T^{N,n}+D_T^{N,n}\right]
\le C\mathcal W_n.
\]
Lemma~\ref{lem:logarithmic-energy-weight-growth} gives
\(\mathcal W_n\le C4^n\), which proves the proposition.
\end{proof}

For \(n\ge0\), let \(\mathcal M^{N,n}\) denote the
\(H^{-\beta,-m}\)-valued martingale defined by
\begin{equation}\label{eq:tightness-martingale-definition}
\langle \mathcal M_t^{N,n},\varphi\rangle
=
\frac1{\sqrt N}\sum_{i=1}^N h_n(i/N)
\int_0^t\nabla\varphi(X_r^i)\cdot\dd B_r^i ,
\end{equation}
for \(\varphi\in C_c^\infty(\R^d)\).

\begin{proposition}\label{prop:weighted-tightness}
For each fixed \(n\ge0\),
the laws of the pairs \((Y^{N,n},\mathcal M^{N,n})_{N\ge1}\) are tight in
\(E\times E\). The full hierarchy and martingale family
\[
\bigl((Y^{N,n})_{n\ge0},(\mathcal M^{N,n})_{n\ge0}\bigr)
\]
are therefore tight in \(E^{\mathbb N_0}\times E^{\mathbb N_0}\).
\end{proposition}

\begin{proof}
Fix \(n\ge0\); constants below may depend on \(T\) and \(n\), but not on
\(N\). By the Arzel\`a--Ascoli tightness criterion
\cite{kelley2017topology}, it is enough to prove fixed-time tightness and
uniform equicontinuity in probability.

For fixed-time tightness, since \(\beta_\ast>d/2+1\),
\eqref{eq:delta-gradient-sobolev-bound}, with \(\alpha=\beta_\ast\) and
\(m=0\), gives
\[
\sup_x\|\nabla\delta_x\|_{H^{-\beta_\ast}}<\infty.
\]
Thus the stochastic integral in
\eqref{eq:tightness-martingale-definition} defines a continuous
\(H^{-\beta_\ast}\)-valued process. The time integral of the squared
Hilbert--Schmidt norm of its integrand over \([s,t]\) satisfies
\[
\begin{aligned}
&\frac1N\sum_{i=1}^N h_n(i/N)^2
\sum_{a=1}^d\int_s^t
\|\partial_a\delta_{X_r^i}\|_{H^{-\beta_\ast}}^2\,\dd r\\
&\qquad\le
C|t-s|\frac1N\sum_{i=1}^N h_n(i/N)^2
\le C|t-s|\|h_n\|_{L^2(0,1)}^2
\le C_n|t-s|,
\end{aligned}
\]
where the second inequality follows because \(h_n^2\) is nonincreasing.
The Hilbert-valued BDG inequality \cite{da2014stochastic} and
Proposition~\ref{prop:weighted-apriori} therefore give
\[
\sup_N\mathbb E\sup_{t\le T}
\left(
\|Y_t^{N,n}\|_{H^{-\beta_\ast}}^2
+\|\mathcal M_t^{N,n}\|_{H^{-\beta_\ast}}^2
\right)
\le C_n.
\]
By the compact embedding \eqref{eq:weighted-rellich-embedding} and Markov's
inequality, the fixed-time laws of
\((Y_t^{N,n},\mathcal M_t^{N,n})\) are tight in
\(H^{-\beta,-m}\times H^{-\beta,-m}\) for every \(t\in[0,T]\).

For time equicontinuity, the same
Hilbert--Schmidt estimate, the BDG inequality, and the continuous embedding in
\eqref{eq:weighted-rellich-embedding} give, for every \(p>1\),
\begin{equation}\label{eq:hilbert-martingale-modulus-bound}
\mathbb E
\|\mathcal M_t^{N,n}-\mathcal M_s^{N,n}\|_{H^{-\beta,-m}}^{2p}
\le
C_{p,n}|t-s|^p,
\qquad 0\le s<t\le T.
\end{equation}
Kolmogorov's criterion yields uniform equicontinuity in probability of
\((\mathcal M^{N,n})_{N\ge1}\) in \(H^{-\beta,-m}\).

For \(Y^{N,n}\), the weak finite-\(N\) equation in \(H^{-\beta}\) is
\begin{equation}\label{eq:tightness-drift-martingale-decomposition}
Y_t^{N,n}-Y_s^{N,n}
=
\int_s^t \mathcal D_r^{N,n}\,\dd r
+
\mathcal M_t^{N,n}-\mathcal M_s^{N,n},
\end{equation}
where, for \(\varphi\in C_c^\infty(\R^d)\),
\[
\begin{aligned}
\langle \mathcal D_r^{N,n},\varphi\rangle
&=
\left\langle Y_r^{N,n},
\frac12\Delta\varphi
+
\bigl(b(r,\cdot)+K*\bar\rho_r\bigr)\cdot\nabla\varphi
\right\rangle\\
&\quad+
\frac1{\sqrt N}\sum_{i=1}^N h_n(i/N)
\nabla\varphi(X_r^i)\cdot\Delta_r^i.
\end{aligned}
\]
Under Assumption~A, the derivatives of
\(b(r,\cdot)+K*\bar\rho_r\) are uniformly bounded. Since
\(\beta>\beta_\ast+2\), the diffusion and transport parts map
\(H^{-\beta_\ast}\) continuously into \(H^{-\beta}\). In addition,
\eqref{eq:delta-gradient-sobolev-bound}, with \(\alpha=\beta\) and \(m=0\),
bounds the interaction-error term. Hence
\[
\|\mathcal D_r^{N,n}\|_{H^{-\beta}}
\le
C\left(
\|Y_r^{N,n}\|_{H^{-\beta_\ast}}
+\frac1{\sqrt N}\sum_{i=1}^N|h_n(i/N)|\,|\Delta_r^i|
\right).
\]
Consequently, Proposition~\ref{prop:weighted-apriori} and
Lemma~\ref{lem:weighted-J} yield
\begin{equation}\label{eq:tightness-drift-L2-bound}
\sup_N\mathbb E\int_0^T
\|\mathcal D_r^{N,n}\|_{H^{-\beta}}^2\,\dd r
\le C_n.
\end{equation}
For \(R>0\), let
\[
\Omega_R^N
:=
\left\{
\int_0^T
\|\mathcal D_r^{N,n}\|_{H^{-\beta}}^2\,\dd r
\le R
\right\}.
\]
By \eqref{eq:tightness-drift-L2-bound},
\(\sup_N\mathbb P((\Omega_R^N)^c)\le C_n/R\), while on
\(\Omega_R^N\),
\[
\left\|
\int_s^t\mathcal D_r^{N,n}\,\dd r
\right\|_{H^{-\beta}}
\le
R^{1/2}|t-s|^{1/2}.
\]
The continuous embedding \(H^{-\beta}\hookrightarrow H^{-\beta,-m}\), dual
to \(H^{\beta,m}\hookrightarrow H^\beta\), shows that the drift integrals are
uniformly equicontinuous in probability in \(H^{-\beta,-m}\).
Equations~\eqref{eq:hilbert-martingale-modulus-bound} and
\eqref{eq:tightness-drift-martingale-decomposition} then give the same
property for the pairs \((Y^{N,n},\mathcal M^{N,n})_{N\ge1}\). Combined with
fixed-time tightness, this proves tightness in \(E\times E\).

For the countable family, given \(\varepsilon>0\), choose for each \(n\ge0\)
a compact set
\(\mathsf K_n\subset E\times E\) such that
\[
\sup_N\mathbb P\{(Y^{N,n},\mathcal M^{N,n})\notin \mathsf K_n\}
\le \varepsilon 2^{-n-1}.
\]
Then \(\mathsf K:=\prod_{n\ge0}\mathsf K_n\) is compact in
\((E\times E)^{\mathbb N_0}\), and
\[
\sup_N\mathbb P\{((Y^{N,n},\mathcal M^{N,n}))_{n\ge0}\notin \mathsf K\}
\le
\sum_{n\ge0}
\sup_N\mathbb P\{(Y^{N,n},\mathcal M^{N,n})\notin \mathsf K_n\}
\le\varepsilon.
\]
Identifying \((E\times E)^{\mathbb N_0}\) with
\(E^{\mathbb N_0}\times E^{\mathbb N_0}\), this proves the stated
countable-product tightness.
\end{proof}

Since \(E^{\mathbb N_0}\times E^{\mathbb N_0}\) is Polish,
Proposition~\ref{prop:weighted-tightness}, Prokhorov's theorem, and the
classical Skorokhod representation theorem
\cite[Theorem~2.6.2]{hofmanova2018spde} yield the following statement.

\begin{corollary}[Skorokhod representation]
\label{cor:subsequential-skorokhod}
Let \((N_j)_{j\ge1}\) be any sequence with \(N_j\to\infty\). There exist a
further subsequence, still denoted by \((N_j)\), a complete stochastic basis
\((\widetilde\Omega,\widetilde{\mathcal F},
(\widetilde{\mathcal F}_t)_{t\in[0,T]},\widetilde{\mathbb P})\), and
\(E^{\mathbb N_0}\times E^{\mathbb N_0}\)-valued random elements
\[
\left(
(\widetilde Y^{N_j,n})_{n\ge0},
(\widetilde{\mathcal M}^{N_j,n})_{n\ge0}
\right),
\qquad
\left(
(\widetilde Y^n)_{n\ge0},
(\widetilde{\mathcal M}^n)_{n\ge0}
\right),
\]
with the following properties:
\begin{enumerate}
\item[(i)] For every \(j\ge1\),
\[
\Law\left(
(\widetilde Y^{N_j,n})_{n\ge0},
(\widetilde{\mathcal M}^{N_j,n})_{n\ge0}
\right)
=
\Law\left(
(Y^{N_j,n})_{n\ge0},
(\mathcal M^{N_j,n})_{n\ge0}
\right)
\!.
\]

\item[(ii)] We have
\begin{equation}\label{eq:skorokhod-countable-convergence}
\left(
(\widetilde Y^{N_j,n})_{n\ge0},
(\widetilde{\mathcal M}^{N_j,n})_{n\ge0}
\right)
\longrightarrow
\left(
(\widetilde Y^n)_{n\ge0},
(\widetilde{\mathcal M}^n)_{n\ge0}
\right)
\quad \widetilde{\mathbb P}\text{-a.s. in }
E^{\mathbb N_0}\times E^{\mathbb N_0}.
\end{equation}

\item[(iii)] For every \(q\in(0,1/4)\),
\[
\sum_{n=0}^\infty
q^n\widetilde{\mathbb E}\sup_{t\le T}
\|\widetilde Y_t^n\|_{H^{-\beta_\ast}}^2
<\infty.
\]
This follows directly from \eqref{eq:weighted-apriori-bound},
\textup{(i)--(ii)}, Fatou's lemma, and the lower semicontinuity on \(E\) of
\(z\mapsto\sup_{t\le T}\|z_t\|_{H^{-\beta_\ast}}^2\). Together with
continuity in \(E\), the same bound gives weakly continuous paths in
\(H^{-\beta_\ast}\) for every \(\widetilde Y^n\).
\end{enumerate}
\end{corollary}

\section{Finite-particle identities and limits of the linear terms}
\label{sec:identification-inputs}

In this section, we prepare the term-by-term passage to the limit in the
finite-\(N\) martingale equation. We first derive the equation and split its
interaction term into a leading mean-field part and a nonlinear remainder. We
then identify the limit of the leading interaction, compute the limiting
martingale covariations, and establish the separate limits of the initial
fluctuations and martingale coordinates. The nonlinear remainder will be
handled by conditional replacement.

\subsection{Finite-N martingale identity}\label{subsec:martingale-formulation}

In this subsection, we derive the weak finite-particle equation and split its
interaction term into two parts. The first lemma also shows that the boundary
term contributed by the first particle vanishes as \(N\to\infty\).

\begin{lemma}\label{lem:finiteN-mp}
Fix \(n\ge0\) and let \(\varphi\in C_c^\infty(\R^d)\). Set
\[
M_t^{N,n}(\varphi)
:=
\frac1{\sqrt N}\sum_{i=1}^N h_n(i/N)
\int_0^t\nabla\varphi(X_s^i)\cdot\dd B_s^i.
\]
Then
\begin{align}
\langle Y_t^{N,n},\varphi\rangle
&=
\langle Y_0^{N,n},\varphi\rangle
+
\int_0^t
\langle Y_s^{N,n},L_s\varphi\rangle\,\dd s\notag\\
&\quad+
\int_0^t I_s^{N,n}(\varphi)\,\dd s
+
\int_0^t
\mathcal B_s^{N,n}(\varphi)\,\dd s
+
M_t^{N,n}(\varphi),
\label{eq:finiteN-martingale-formulation}
\end{align}
where
\(L_s\varphi=\frac12\Delta\varphi+
\bigl(b(s,\cdot)+K*\bar\rho_s\bigr)\cdot\nabla\varphi\),
\begin{equation}\label{eq:interaction-term}
I_s^{N,n}(\varphi)
=
\frac1{\sqrt N}
\sum_{i=2}^N
h_n(i/N)
\nabla\varphi(X_s^i)\cdot
(K*(\mu_s^{i-1}-\bar\rho_s))(X_s^i),
\end{equation}
\begin{equation}\label{eq:first-particle-error}
\mathcal B_s^{N,n}(\varphi)
=
-\frac{h_n(1/N)}{\sqrt N}
\nabla\varphi(X_s^1)\cdot(K*\bar\rho_s)(X_s^1).
\end{equation}
Moreover, as \(N\to\infty\),
\begin{equation}\label{eq:first-particle-error-vanishes}
\mathbb E\int_0^T |\mathcal B_s^{N,n}(\varphi)|\,\dd s\to0.
\end{equation}
\end{lemma}

\begin{proof}
For \(i\ge2\),
\[
\Delta_s^i=(K*(\mu_s^{i-1}-\bar\rho_s))(X_s^i),
\qquad
\Delta_s^1=-(K*\bar\rho_s)(X_s^1).
\]
Substituting these identities into \eqref{eq:weak-YN} gives
\eqref{eq:finiteN-martingale-formulation}.
The boundary estimate \eqref{eq:first-particle-error-vanishes} follows from
\eqref{eq:first-particle-error}, the boundedness of \(K\) and \(\nabla\varphi\),
and Lemma~\ref{lem:elementary-weights}.
\end{proof}

To separate the part that survives in the limit from the nonlinear remainder,
define
\[
F_\varphi[\gamma](x)
:=
\nabla\varphi(x)\cdot(K*\gamma)(x).
\]
Using this notation, split the interaction term \eqref{eq:interaction-term}
into its mean-field average and a nonlinear remainder:
\[
I_s^{N,n}(\varphi)
=
I_{s,0}^{N,n}(\varphi)
+
\mathcal R_s^{N,n}(\varphi),
\]
where
\begin{equation}\label{eq:leading-interaction-term}
I_{s,0}^{N,n}(\varphi)
=
\frac1{\sqrt N}
\sum_{i=2}^N
h_n(i/N)
\left\langle
\bar\rho_s,
F_\varphi[\mu_s^{i-1}-\bar\rho_s]
\right\rangle,
\end{equation}
and
\begin{equation}\label{eq:nonlinear-remainder}
\mathcal R_s^{N,n}(\varphi)
=
\frac1{\sqrt N}
\sum_{i=2}^N
h_n(i/N)
\left\langle
\delta_{X_s^i}-\bar\rho_s,
F_\varphi[\mu_s^{i-1}-\bar\rho_s]
\right\rangle.
\end{equation}
The first term is linear in the predecessor fluctuation and can be rewritten
exactly. The second is the nonlinear remainder to be removed by conditional
replacement. Reindexing the lower-triangular sum in the first term produces
the weights
\[
(\mathsf T_N h_n)(j/N)
:=
\sum_{i=j+1}^N\frac{h_n(i/N)}{i-1},
\qquad 1\le j\le N-1,
\qquad
(\mathsf T_N h_n)(1):=0.
\]

\begin{lemma}\label{lem:exact-summation}
Fix \(n\ge0\) and let \(\varphi\in C_c^\infty(\R^d)\). Then
\begin{equation}\label{eq:leading-identity}
I_{s,0}^{N,n}(\varphi)
=
\left\langle
\bar\rho_s,
(K*Y_s^N(\mathsf T_N h_n))\cdot\nabla\varphi
\right\rangle.
\end{equation}
\end{lemma}

\begin{proof}
Using \eqref{eq:leading-interaction-term} and
\[
  \mu_s^{i-1}-\bar\rho_s
  =
  \frac1{i-1}\sum_{j<i}(\delta_{X_s^j}-\bar\rho_s),
\]
we get
\[
\begin{aligned}
I_{s,0}^{N,n}(\varphi)
&=
\frac1{\sqrt N}
\sum_{i=2}^N
h_n(i/N)
\frac1{i-1}
\sum_{j<i}
\left\langle
\bar\rho_s,
(K*(\delta_{X_s^j}-\bar\rho_s))\cdot\nabla\varphi
\right\rangle\\
&=
\frac1{\sqrt N}
\sum_{j=1}^{N-1}
\left(
\sum_{i=j+1}^N\frac{h_n(i/N)}{i-1}
\right)
\left\langle
\bar\rho_s,
(K*(\delta_{X_s^j}-\bar\rho_s))\cdot\nabla\varphi
\right\rangle\\
&=
\frac1{\sqrt N}
\sum_{j=1}^{N-1}
(\mathsf T_N h_n)(j/N)
\left\langle
\bar\rho_s,
(K*(\delta_{X_s^j}-\bar\rho_s))\cdot\nabla\varphi
\right\rangle\\
&=
\left\langle
\bar\rho_s,
(K*Y_s^N(\mathsf T_N h_n))\cdot\nabla\varphi
\right\rangle.
\end{aligned}
\]
\end{proof}

\subsection{Limit of the leading interaction term}

In this subsection, we identify the limit of the leading interaction term.
Lemma~\ref{lem:exact-summation} expresses this term through
\(Y^N(\mathsf T_N h_n)\), whereas the next hierarchy level is
\(Y^{N,n+1}=Y^N(h_{n+1})\). It therefore remains to replace
\(\mathsf T_N h_n\) by \(h_{n+1}\). The next lemma proves this replacement
after testing and time integration.

\begin{lemma}\label{lem:TN-error-weighted-lln}
Fix \(n\ge0\). Let \((s,x)\mapsto\Gamma_s(x)\) be jointly Borel measurable
and bounded:
\[
\sup_{s\le T}\|\Gamma_s\|_\infty<\infty.
\]
Then, as \(N\to\infty\),
\[
\mathbb E\int_0^T
\left|
\left\langle
Y_s^N(\mathsf T_N h_n)-Y_s^{N,n+1},\Gamma_s
\right\rangle
\right|\,\dd s
\to0.
\]
\end{lemma}

\begin{proof}
Use the jointly measurable conditional-law kernels constructed by time-space
disintegration in Lemma~\ref{lem:paper1-estimates}\textup{(ii)}. By the
definition of the weighted fluctuation fields,
\[
\begin{aligned}
&\left\langle
Y_s^N(\mathsf T_N h_n)-Y_s^{N,n+1},\Gamma_s
\right\rangle\\
&\quad=
\frac1{\sqrt N}\sum_{i=1}^N
\bigl((\mathsf T_N h_n)(i/N)-h_{n+1}(i/N)\bigr)
\left(
\Gamma_s(X_s^i)-\langle\bar\rho_s,\Gamma_s\rangle
\right).
\end{aligned}
\]
For a.e.\ \(s\), decompose
\[
\Gamma_s(X_s^i)-\langle\bar\rho_s,\Gamma_s\rangle
=
\left(
\Gamma_s(X_s^i)-\langle\nu_s^i,\Gamma_s\rangle
\right)
+
\left(
\langle\nu_s^i-\bar\rho_s,\Gamma_s\rangle
\right).
\]
The first part is a martingale-difference sequence in \(i\). By orthogonality,
the squared \(L^2\)-norm of its weighted sum is bounded by
\[
\frac{C}{N}
\sum_{i=1}^N
\bigl|(\mathsf T_N h_n)(i/N)-h_{n+1}(i/N)\bigr|^2
\to0
\]
by Lemma~\ref{lem:TN-error-summability}. Hence this contribution tends
to zero in \(L^1\).

For the predictable bias part,
\[
\left|
\langle\nu_s^i-\bar\rho_s,\Gamma_s\rangle
\right|
\le
\|\Gamma_s\|_\infty
\|\nu_s^i-\bar\rho_s\|_{\mathrm{TV}}.
\]
For \(i\ge2\), by Pinsker's inequality and
Lemma~\ref{lem:paper1-estimates}\textup{(iii)},
\[
\mathbb E\|\nu_s^i-\bar\rho_s\|_{\mathrm{TV}}
\le
\left(
2\mathbb E H(\nu_s^i\mid\bar\rho_s)
\right)^{1/2}
\le
\frac{C_T}{\sqrt{i-1}}.
\]
For \(i=1\), the same quantity is uniformly bounded by \(2\). Thus the
\(L^1\)-norm of the predictable-bias contribution is bounded by
\[
\begin{aligned}
C\biggl(&
\frac{|(\mathsf T_N h_n)(1/N)-h_{n+1}(1/N)|}{\sqrt N}\\
&+
\frac1{\sqrt N}
\sum_{i=2}^N
\frac{|(\mathsf T_N h_n)(i/N)-h_{n+1}(i/N)|}{\sqrt{i-1}}
\biggr)
\longrightarrow0.
\end{aligned}
\]
Both terms vanish by the \(L^2\) and predecessor-weighted \(L^1\) estimates in
Lemma~\ref{lem:TN-error-summability}; for the second, also use
\(1/\sqrt{i-1}\le \sqrt2/\sqrt i\) for \(i\ge2\). The bounds are uniform in
\(s\), so integration over \([0,T]\) proves the claim.
\end{proof}

To apply this replacement to the leading interaction, for
\(K^\vee(x):=K(-x)\) and \(\varphi\in C_c^\infty(\R^d)\), set
{
\[
\Gamma_s^\varphi(y)
:=
\int_{\R^d}
\nabla\varphi(x)\cdot K(x-y)\,\bar\rho_s(\dd x)
=
\sum_{a=1}^d
K_a^\vee*((\partial_a\varphi)\bar\rho_s)(y).
\]
Since \(K\in C_b^\infty\cap L^1\), the Gagliardo--Nirenberg interpolation
inequality gives \(K\in H^r\) for every \(r\ge0\). For
\(\mu\in\mathcal P(\R^d)\) and bounded \(f\), Minkowski's inequality gives
\begin{equation}\label{eq:kernel-measure-convolution-bound}
\|K^\vee*(f\mu)\|_{H^r}
\le
\|K\|_{H^r}\|f\mu\|_{\mathrm{TV}}
\le
\|K\|_{H^r}\|f\|_\infty,
\qquad r\ge0.
\end{equation}
Applying \eqref{eq:kernel-measure-convolution-bound} componentwise with
\(\mu=\bar\rho_s\) and \(f=\partial_a\varphi\) gives
\begin{equation}\label{eq:gamma-unweighted-bound}
\sup_{s\le T}\|\Gamma_s^\varphi\|_{H^r}
\le
C_\varphi\|K\|_{H^r}.
\end{equation}
Since \(K\) is bounded, we also have
\(\sup_{s\le T}\|\Gamma_s^\varphi\|_\infty<\infty\).
}

We can now identify the leading interaction along every convergent
subsequence.

\begin{lemma}\label{lem:leading-limit}
Fix \(n\ge0\) and let \(\varphi\in C_c^\infty(\R^d)\). Then, as
\(N\to\infty\),
\begin{equation}\label{eq:leading-interaction-approximation}
\mathbb E\sup_{t\le T}
\left|
\int_0^t I_{s,0}^{N,n}(\varphi)\,\dd s
-
\int_0^t
\langle Y_s^{N,n+1},\Gamma_s^\varphi\rangle\,\dd s
\right|
\longrightarrow0.
\end{equation}
Consequently, along any subsequence for which
\(Y^{N,n+1}\LawConv Y^{n+1}\) in \(E\),
\[
  \left(
  Y^{N,n+1},
  \left(\int_0^t I_{s,0}^{N,n}(\varphi)\,\dd s\right)_{t\in[0,T]}
  \right)
  \LawConv
  \left(
  Y^{n+1},
  \left(\int_0^t
  \langle Y_s^{n+1},\Gamma_s^\varphi\rangle\,\dd s
  \right)_{t\in[0,T]}
  \right)
\]
weakly in \(E\times C([0,T];\R)\).  The second coordinate of the limit is
equivalently
\[
\left(\int_0^t
\left\langle
\bar\rho_s,(K*Y_s^{n+1})\cdot\nabla\varphi
\right\rangle\,\dd s\right)_{t\in[0,T]}.
\]
\end{lemma}

\begin{proof}
{
Since \(K\in H^r\) for every \(r\ge0\), we have
\(K*z\in C_b^\infty(\R^d;\R^d)\) for every
\(z\in H^{-\beta_\ast}\). Thus the interaction is an ordinary integral, and
the definition of \(\Gamma_s^\varphi\) gives
\[
\left\langle
\bar\rho_s,
(K*z)\cdot\nabla\varphi
\right\rangle
=
\langle z,\Gamma_s^\varphi\rangle.
\]
}
Hence \eqref{eq:leading-identity} gives
\[
I_{s,0}^{N,n}(\varphi)
=
\left\langle Y_s^N(\mathsf T_N h_n),\Gamma_s^\varphi\right\rangle.
\]
Since \(\Gamma^\varphi\) is bounded, Lemma~\ref{lem:TN-error-weighted-lln}
implies
\[
\begin{aligned}
&\mathbb E\sup_{t\le T}
\left|
\int_0^t I_{s,0}^{N,n}(\varphi)\,\dd s
-
\int_0^t
\langle Y_s^{N,n+1},\Gamma_s^\varphi\rangle\,\dd s
\right|\\
&\qquad\le
\mathbb E\int_0^T
\left|
\left\langle
Y_s^N(\mathsf T_N h_n)-Y_s^{N,n+1},\Gamma_s^\varphi
\right\rangle
\right|\,\dd s
\longrightarrow0.
\end{aligned}
\]
This proves \eqref{eq:leading-interaction-approximation}.

It remains to pass the term containing \(Y^{N,n+1}\) to the limit. Choose
\(\chi\in C_c^\infty(\R^d)\) with \(\chi=1\) on the unit ball, set
\(\chi_R(x)=\chi(x/R)\), and write
\(\Gamma_{s,R}^\varphi=\chi_R\Gamma_s^\varphi\). For every fixed \(R\),
\[
z\longmapsto
\left(\int_0^t
\langle z_s,\Gamma_{s,R}^\varphi\rangle\,\dd s\right)_{t\in[0,T]}
\]
is continuous from \(E\) to \(C([0,T];\R)\), since
{the weak continuity of \(s\mapsto\bar\rho_s\) and the bounded
smoothness of \(K\) imply that \(s\mapsto\Gamma_{s,R}^\varphi\) is continuous
in \(H^{\beta,m}\).}

By the \(H^{-\beta_\ast}\) estimate \eqref{eq:weighted-apriori-bound} and lower
semicontinuity,
\[
\sup_N\mathbb E\sup_{s\le T}
\|Y_s^{N,n+1}\|_{H^{-\beta_\ast}}^2
+
\mathbb E\sup_{s\le T}
\|Y_s^{n+1}\|_{H^{-\beta_\ast}}^2
<\infty.
\]
Hence, for either \(Z=Y^{N,n+1}\) or \(Z=Y^{n+1}\),
\[
\mathbb E\sup_{t\le T}
\left|
\int_0^t
\langle Z_s,(1-\chi_R)\Gamma_s^\varphi\rangle\,\dd s
\right|
\le
C_n\int_0^T
\|(1-\chi_R)\Gamma_s^\varphi\|_{H^{\beta_\ast}}\,\dd s.
\]
For each \(s\), multiplication by the cutoff gives
\((1-\chi_R)\Gamma_s^\varphi\to0\) in \(H^{\beta_\ast}\), and these multiplier
operators are uniformly bounded in \(R\). Thus the right-hand side tends to
zero by \eqref{eq:gamma-unweighted-bound} and dominated convergence. The continuous
mapping theorem for the truncated functionals, followed by this uniform tail
estimate, proves the joint convergence of
\[
\left(
Y^{N,n+1},
\left(\int_0^t
\langle Y_s^{N,n+1},\Gamma_s^\varphi\rangle\,\dd s
\right)_{t\in[0,T]}
\right).
\]
Together with \eqref{eq:leading-interaction-approximation}, this gives the
joint convergence in the statement. The equivalent form of the limiting
interaction follows from the first identity in the proof.
\end{proof}

\subsection{Initial fluctuations and martingale noise}

In this subsection, we identify the initial fluctuations and martingale noise
separately. We first use the conditional-law estimates and a logarithmic
Riemann sum to compute the limiting martingale covariations. We then apply the
weighted central limit theorem to the initial fluctuations and the martingale
central limit theorem to the martingale coordinates.

\begin{lemma}\label{lem:weighted-martingale-covariance}
Fix \(n,k\ge0\) and let \(\varphi,\psi\in C_c^\infty(\R^d)\). Then
\begin{align}
\label{eq:finiteN-martingale-covariance}
\left\langle
M^{N,n}(\varphi),M^{N,k}(\psi)
\right\rangle_t
&=
\frac1N
\sum_{i=1}^N
h_n(i/N)h_k(i/N)
\int_0^t
\nabla\varphi(X_s^i)\cdot
\nabla\psi(X_s^i)\,\dd s.
\end{align}
Moreover, as \(N\to\infty\),
\begin{equation}\label{eq:martingale-covariance-limit}
\left\langle
M^{N,n}(\varphi),M^{N,k}(\psi)
\right\rangle_t
\to
\frac{(n+k)!}{n!\,k!}
\int_0^t
\left\langle
\bar\rho_s,
\nabla\varphi\cdot
\nabla\psi
\right\rangle \dd s
\end{equation}
in \(L^1\), uniformly in \(t\in[0,T]\).
\end{lemma}

\begin{proof}
The identity \eqref{eq:finiteN-martingale-covariance} follows from independence
of the Brownian motions. Set \(g=\nabla\varphi\cdot\nabla\psi\), and use the
jointly measurable conditional-law kernels constructed in
Lemma~\ref{lem:paper1-estimates}\textup{(ii)}. For a.e.\ \(s\), write
\[
g(X_s^i)-\langle\bar\rho_s,g\rangle
=
\bigl(g(X_s^i)-\langle\nu_s^i,g\rangle\bigr)
+
\langle\nu_s^i-\bar\rho_s,g\rangle.
\]
The first term is a martingale difference in \(i\). Hence
\[
\mathbb E
\left|
\frac1N\sum_{i=1}^N h_n(i/N)h_k(i/N)
\bigl(g(X_s^i)-\langle\nu_s^i,g\rangle\bigr)
\right|
\le
\left(
\frac{C}{N^2}\sum_{i=1}^N|h_n(i/N)h_k(i/N)|^2
\right)^{1/2}
\le \frac{C_{n,k}}{\sqrt N}.
\]
For the second term, the trivial bound for \(i=1\), Pinsker's inequality, and
Lemma~\ref{lem:paper1-estimates}\textup{(iii)} give
\[
\mathbb E
\left|
\frac1N\sum_{i=1}^N h_n(i/N)h_k(i/N)
\langle\nu_s^i-\bar\rho_s,g\rangle
\right|
\le
\frac{C}{N}
\left(
|h_n(1/N)h_k(1/N)|
+
\sum_{i=2}^N
\frac{|h_n(i/N)h_k(i/N)|}{\sqrt{i-1}}
\right)
\longrightarrow0.
\]
Both bounds are uniform in \(s\), and the last limit follows from
Lemma~\ref{lem:product-weight-bounds}. Using
\(\sup_{t\le T}|\int_0^t f_s\,\dd s|\le\int_0^T|f_s|\,\dd s\), we obtain
\[
\mathbb E\sup_{t\le T}
\left|
\frac1N\sum_{i=1}^N h_n(i/N)h_k(i/N)
\int_0^t
\bigl(g(X_s^i)-\langle\bar\rho_s,g\rangle\bigr)\,\dd s
\right|
\longrightarrow0.
\]

Finally, Lemma~\ref{lem:riemann-weights} gives
\[
\frac1N\sum_{i=1}^N h_n(i/N)h_k(i/N)
\longrightarrow
\frac{(n+k)!}{n!\,k!}.
\]
Since \(g\) is bounded, the corresponding deterministic mean terms converge
uniformly on \([0,T]\). This proves
\eqref{eq:martingale-covariance-limit}.
\end{proof}

Using this covariation limit, we can now identify the martingale coordinates.
The initial coordinates follow from a weighted Lindeberg--Feller argument, so
we state the two limits together.

\begin{proposition}\label{prop:gaussian-input-convergence}
Let \(n_1,\ldots,n_q,k_1,\ldots,k_r\ge0\), and let
\(\varphi_1,\ldots,\varphi_q,\psi_1,\ldots,\psi_r\in
C_c^\infty(\R^d)\). Then, as \(N\to\infty\), the following convergences hold.
\begin{enumerate}
\item[(i)] The initial vector
\[
\bigl(\langle Y_0^{N,n_a},\varphi_a\rangle\bigr)_{a=1}^q
\]
converges weakly in \(\R^q\) to the centered Gaussian vector with covariance
\eqref{eq:limit-initial-covariance}.

\item[(ii)] The martingale vector
\[
\bigl(M^{N,k_b}(\psi_b)\bigr)_{b=1}^r
\]
converges weakly in \(C([0,T];\R)^r\) to the centered continuous Gaussian
martingale with covariations \eqref{eq:limit-martingale-covariance}.
\end{enumerate}
\end{proposition}

\begin{proof}
The initial vector is the normalized sum of independent centered vectors
\[
\frac1{\sqrt N}\sum_{i=1}^N
\left(
h_{n_a}(i/N)
\bigl(\varphi_a(X_0^i)-\langle\bar\rho_0,\varphi_a\rangle\bigr)
\right)_{a=1}^q.
\]
For \(1\le a,c\le q\), its covariance entry is
\[
\left(
\frac1N\sum_{i=1}^N
h_{n_a}(i/N)h_{n_c}(i/N)
\right)
\left[
\langle\bar\rho_0,\varphi_a\varphi_c\rangle
-
\langle\bar\rho_0,\varphi_a\rangle
\langle\bar\rho_0,\varphi_c\rangle
\right],
\]
which converges to the corresponding entry in
\eqref{eq:limit-initial-covariance} by
Lemma~\ref{lem:riemann-weights}. Moreover,
Lemma~\ref{lem:elementary-weights} and the boundedness of the test functions
show that the norm of every normalized summand is bounded by
\[
\frac{C}{\sqrt N}
\max_{1\le i\le N}\max_{1\le a\le q}h_{n_a}(i/N)
\longrightarrow0.
\]
Thus the Lindeberg condition holds, and the multivariate
Lindeberg--Feller theorem proves \textup{(i)}.

For the martingale vector, its bracket matrix is given by
\eqref{eq:finiteN-martingale-covariance} and converges in probability,
uniformly on \([0,T]\), to the deterministic matrix in
\eqref{eq:limit-martingale-covariance} by
Lemma~\ref{lem:weighted-martingale-covariance}. Since the martingales are
continuous, the vector martingale central limit theorem
\cite[Theorem~VIII.3.11]{jacodshiryaev2003limit} proves \textup{(ii)}.
\end{proof}

\section{Conditional replacement and limit identification}
\label{sec:conditional-replacement}
\label{sec:limit-identification}

In this section, we complete the identification of subsequential limits. We
first use conditional replacement to show that the nonlinear remainder
vanishes. We then combine this estimate with the limits established in
Section~\ref{sec:identification-inputs} to derive the equation for every
subsequential limit.

\subsection{Conditional replacement}

Fix \(n\ge0\) and \(\varphi\in C_c^\infty(\R^d)\). For a signed measure
\(\gamma\), set
\[
K_\varphi(x,y):=\nabla\varphi(x)\cdot K(x-y),
\qquad
F_\varphi[\gamma](x)
:=\nabla\varphi(x)\cdot(K*\gamma)(x)
=\int K_\varphi(x,y)\,\gamma(\dd y).
\]
The nonlinear remainder introduced in
\eqref{eq:nonlinear-remainder} is
\[
\mathcal R_s^{N,n}(\varphi)
=
\frac1{\sqrt N}\sum_{i=2}^N h_n(i/N)
\left\langle
\delta_{X_s^i}-\bar\rho_s,
F_\varphi[\mu_s^{i-1}-\bar\rho_s]
\right\rangle.
\]
We show that its time integral converges to zero uniformly on \([0,T]\).

For \(i\ge2\), let \(f_s^j\), \(2\le j\le i\), be the jointly Borel kernels
obtained in Lemma~\ref{lem:paper1-estimates}\textup{(ii)} by disintegrating
\(\dd s\otimes\Law(X_s^{1:j})\), and set
\[
\nu_s^1:=\Law(X_s^1),
\qquad
\nu_s^j(\omega):=f_s^j(X_s^{1:j-1}(\omega),\cdot),
\quad 2\le j\le i,
\]
\[
\bar\nu_s^{i-1}:=\frac1{i-1}\sum_{j<i}\nu_s^j.
\]
The resulting maps \((s,\omega)\mapsto\nu_s^j(\omega)\) are jointly
measurable, and \(\nu_s^j\) is a version of
\(\Law(X_s^j\mid\mathscr G_s^{j-1})\) in the
\(\dd s\otimes\mathbb P\)-a.s.\ sense. Consequently, for a.e.\ \(s\),
\[
  \mathbb E\left[
  \langle \delta_{X_s^i}-\nu_s^i,\psi\rangle
  \mid \mathscr G_s^{i-1}
  \right]=0
\quad\mathbb P\text{-a.s.}
\]
whenever \(\psi\) is \(\mathscr G_s^{i-1}\)-measurable and bounded.
All such identities below are used under time integration.

We split the two fluctuation factors in \(\mathcal R^{N,n}\) as follows:
\begin{equation}\label{eq:conditional-replacement-decomposition}
\delta_{X_s^i}-\bar\rho_s
=
(\delta_{X_s^i}-\nu_s^i)+(\nu_s^i-\bar\rho_s),
\qquad
\mu_s^{i-1}-\bar\rho_s
=
(\mu_s^{i-1}-\bar\nu_s^{i-1})
+(\bar\nu_s^{i-1}-\bar\rho_s).
\end{equation}
In each line, the first difference is built from martingale differences, while
the second is controlled by the conditional entropy estimates in
Lemma~\ref{lem:paper1-estimates}\textup{(iii)}. The next lemma controls the
difference \(\mu_s^{i-1}-\bar\nu_s^{i-1}\). Its first estimate is used when
this difference is paired with
\(\delta_{X_s^i}-\nu_s^i\); its exponential estimate is used when it is paired
with \(\nu_s^i-\bar\rho_s\).

\begin{lemma}\label{lem:conditional-replacement-estimates}
For \(i\ge2\), uniformly for a.e.\ \(s\in[0,T]\),
\[
\mathbb E
|F_\varphi[\mu_s^{i-1}-\bar\nu_s^{i-1}](X_s^i)|^2
\le C_{\varphi,T}(i-1)^{-1}.
\]
Moreover, if
\[
\widetilde K_{\varphi,s}(x,y)
=
K_\varphi(x,y)
-
\int K_\varphi(z,y)\bar\rho_s(\dd z)
\]
and
\[
\widetilde\xi_s^{i-1}(x)
=
\frac1{i-1}\sum_{j<i}
\left(
\widetilde K_{\varphi,s}(x,X_s^j)
-
\int \widetilde K_{\varphi,s}(x,y)\nu_s^j(\dd y)
\right),
\]
then for every \(\lambda\in\R\),
\[
\mathbb E\log
\int e^{\lambda\widetilde\xi_s^{i-1}(x)}\bar\rho_s(\dd x)
\le
C_\varphi\lambda^2(i-1)^{-1}.
\]
\end{lemma}

\begin{proof}
We first derive a sub-Gaussian estimate that will be used for both conclusions.
For a bounded measurable function \(\phi_s(x,y)\), set
\begin{equation*}
\Xi_s^{i-1,\phi}(x)
=
\frac1{i-1}\sum_{j<i}
\left(
\phi_s(x,X_s^j)-\int \phi_s(x,y)\nu_s^j(\dd y)
\right).
\end{equation*}
The summands are bounded martingale differences with respect to
\((\mathscr G_s^j)_{j<i}\), so
Lemma~\ref{lem:paper1-estimates}\textup{(iv)} gives, for every
\(\lambda\in\R\),
\begin{equation}\label{eq:conditional-xi-exponential}
\mathbb E\int
\exp(\lambda\Xi_s^{i-1,\phi}(x))\bar\rho_s(\dd x)
\le
\exp\left(C_\phi\lambda^2/(i-1)\right),
\end{equation}

\smallskip
\noindent\emph{The \(L^2\) estimate.}
Take \(\phi_s=K_\varphi\) in \eqref{eq:conditional-xi-exponential}. By
definition,
\[
\Xi_s^{i-1,K_\varphi}(x)
=
F_\varphi[\mu_s^{i-1}-\bar\nu_s^{i-1}](x).
\]
To evaluate this predecessor fluctuation at the correlated point \(X_s^i\), set
\[
P_s^{1:i}=\Law(X_s^1,\ldots,X_s^i),
\qquad
Q_s^{1:i}=\Law(X_s^1,\ldots,X_s^{i-1})\otimes\bar\rho_s.
\]
Under \(Q_s^{1:i}\), let \(Z_s^i\) denote the last coordinate. It has law
\(\bar\rho_s\) and is independent of the predecessors. Thus the centered random
variable
\(\Xi_s^{i-1,K_\varphi}(Z_s^i)\) satisfies the sub-Gaussian MGF estimate
\eqref{eq:conditional-xi-exponential}. By the standard equivalence of
sub-Gaussian estimates \cite[Proposition~2.5.2]{vershynin2018highdimensional},
there are constants \(c_\varphi,C_\varphi>0\) such that
\begin{equation*}
\mathbb E_{Q_s^{1:i}}
\exp\left(
c_\varphi(i-1)
\left|
\Xi_s^{i-1,K_\varphi}(Z_s^i)
\right|^2
\right)
\le C_\varphi.
\end{equation*}
The true law is compared with this decoupled law by data processing:
\begin{equation}\label{eq:conditional-data-processing}
H(P_s^{1:i}\mid Q_s^{1:i})\le R_i(s).
\end{equation}
The exponential estimate and the relative-entropy variational inequality
\cite[Proposition~1.4.2]{dupuis2011weakbook} give
\begin{equation}\label{eq:conditional-dv-square-bound}
c_\varphi(i-1)
\mathbb E_{P_s^{1:i}}
\left|\Xi_s^{i-1,K_\varphi}(X_s^i)\right|^2
\le
H(P_s^{1:i}\mid Q_s^{1:i})+\log C_\varphi .
\end{equation}
By \eqref{eq:conditional-data-processing} and
Lemma~\ref{lem:paper1-estimates}\textup{(i)},
\(H(P_s^{1:i}\mid Q_s^{1:i})\le R_i(s)\le C_T(i-1)^{-1}\), and
\eqref{eq:conditional-dv-square-bound}, together with the identity above,
yields
\[
\mathbb E
|F_\varphi[\mu_s^{i-1}-\bar\nu_s^{i-1}](X_s^i)|^2
\le C_{\varphi,T}(i-1)^{-1}.
\]

\smallskip
\noindent\emph{The centered exponential estimate.}
Take \(\phi_s=\widetilde K_{\varphi,s}\) in
\eqref{eq:conditional-xi-exponential}. Since
\[
\widetilde\xi_s^{i-1}(x)
=
\Xi_s^{i-1,\widetilde K_{\varphi,s}}(x),
\]
Jensen's inequality and \eqref{eq:conditional-xi-exponential} give
\[
\mathbb E\log
\int e^{\lambda\widetilde\xi_s^{i-1}(x)}\bar\rho_s(\dd x)
\le
\log\mathbb E
\int e^{\lambda\widetilde\xi_s^{i-1}(x)}\bar\rho_s(\dd x)
\le
C_\varphi\lambda^2(i-1)^{-1}.
\]
\end{proof}

Combining Lemma~\ref{lem:conditional-replacement-estimates} with the estimates
in Lemma~\ref{lem:paper1-estimates}\textup{(iii)}, we now prove that the full
nonlinear remainder vanishes.

\begin{proposition}\label{prop:weighted-conditional-measure-replacement}
Fix \(n\ge0\) and let \(\varphi\in C_c^\infty(\R^d)\). Then, as
\(N\to\infty\),
\[
  \mathbb E\sup_{0\le t\le T}
  \left|
  \int_0^t \mathcal R_s^{N,n}(\varphi)\,\dd s
  \right|
  \to0
\]
and hence \(\int_0^\cdot \mathcal R_s^{N,n}(\varphi)\,\dd s\to0\) in probability in
\(C([0,T];\R)\).
\end{proposition}

\begin{proof}
Set
\[
\chi_s^i:=\delta_{X_s^i}-\nu_s^i,\quad
\theta_s^i:=\nu_s^i-\bar\rho_s,\quad
\zeta_s^{i-1}:=\mu_s^{i-1}-\bar\nu_s^{i-1},\quad
\omega_s^{i-1}:=\bar\nu_s^{i-1}-\bar\rho_s .
\]
By bilinearity,
\[
\mathcal R_s^{N,n}(\varphi)
=A_s^{N,n}(\varphi)+B_s^{N,n}(\varphi)
+C_s^{N,n}(\varphi)+D_s^{N,n}(\varphi),
\]
where
\begin{align*}
A_s^{N,n}(\varphi)
&:=\frac1{\sqrt N}\sum_{i=2}^N h_n(i/N)
\left\langle\chi_s^i,F_\varphi[\zeta_s^{i-1}]\right\rangle,\\
B_s^{N,n}(\varphi)
&:=\frac1{\sqrt N}\sum_{i=2}^N h_n(i/N)
\left\langle\chi_s^i,F_\varphi[\omega_s^{i-1}]\right\rangle,\\
C_s^{N,n}(\varphi)
&:=\frac1{\sqrt N}\sum_{i=2}^N h_n(i/N)
\left\langle\theta_s^i,F_\varphi[\zeta_s^{i-1}]\right\rangle,\\
D_s^{N,n}(\varphi)
&:=\frac1{\sqrt N}\sum_{i=2}^N h_n(i/N)
\left\langle\theta_s^i,F_\varphi[\omega_s^{i-1}]\right\rangle.
\end{align*}
The terms \(A\) and \(B\) contain the conditionally centered measure
\(\chi_s^i\), so they are controlled in \(L^2\) by orthogonality in the
particle index. The term \(C\) pairs \(\theta_s^i\) with the predecessor
fluctuation and requires the exponential estimate of
Lemma~\ref{lem:conditional-replacement-estimates}. The term \(D\) is controlled
directly by the total-variation estimates in
Lemma~\ref{lem:paper1-estimates}\textup{(iii)}.

\smallskip
\noindent\emph{The terms \(A\) and \(B\).}
The functions \(F_\varphi[\zeta_s^{i-1}]\) and
\(F_\varphi[\omega_s^{i-1}]\) are bounded and
\(\mathscr G_s^{i-1}\)-measurable. For any such random function \(f\), the
conditional-law identity gives
\[
\mathbb E\left[
\left|\langle\delta_{X_s^i}-\nu_s^i,f\rangle\right|^2
\,\middle|\,\mathscr G_s^{i-1}
\right]
\le
\int |f(x)|^2\nu_s^i(\dd x).
\]
Consequently, Lemma~\ref{lem:conditional-replacement-estimates} gives
\[
\mathbb E
\left|\left\langle
\chi_s^i,F_\varphi[\zeta_s^{i-1}]
\right\rangle\right|^2
\le C_{\varphi,T}(i-1)^{-1}.
\]
Moreover, the conditional-law bound above and
\(\|F_\varphi[\omega]\|_\infty
\le \|K_\varphi\|_\infty\|\omega\|_{\mathrm{TV}}\) give
\[
\begin{aligned}
\mathbb E
\left|\left\langle
\chi_s^i,F_\varphi[\omega_s^{i-1}]
\right\rangle\right|^2
&\le
\mathbb E\int
|F_\varphi[\omega_s^{i-1}](x)|^2\nu_s^i(\dd x)\\
&\le
\|K_\varphi\|_\infty^2
\mathbb E\|\omega_s^{i-1}\|_{\mathrm{TV}}^2
\le C_{\varphi,T}(i-1)^{-1}.
\end{aligned}
\]
where the last inequality is Lemma~\ref{lem:paper1-estimates}\textup{(iii)}.
The summands in both \(A\) and \(B\) are martingale differences in \(i\).
Orthogonality therefore yields
\[
\mathbb E|A_s^{N,n}|^2+\mathbb E|B_s^{N,n}|^2
\le
\frac{C_{\varphi,T}}N
\sum_{i=2}^N\frac{h_n(i/N)^2}{i-1}.
\]
By the Cauchy--Schwarz inequality in time and the first estimate in
Lemma~\ref{lem:polylog-summability},
\[
\mathbb E\int_0^T
\bigl(|A_s^{N,n}|+|B_s^{N,n}|\bigr)\,\dd s
\xrightarrow[N\to\infty]{}0.
\]

\smallskip
\noindent\emph{The term \(C\).}
Let \(\xi_s^{i-1}=F_\varphi[\zeta_s^{i-1}]\) and center it by
\[
\widetilde\xi_s^{i-1}:=\xi_s^{i-1}-\langle\bar\rho_s,\xi_s^{i-1}\rangle.
\]
This is the centered predecessor fluctuation associated with
\(\widetilde K_{\varphi,s}\) in Lemma~\ref{lem:conditional-replacement-estimates}.
It is bounded and \(\mathscr G_s^{i-1}\)-measurable.
Then the \(i\)-th summand is
\[
\langle\theta_s^i,F_\varphi[\zeta_s^{i-1}]\rangle
=
\langle\nu_s^i,\widetilde\xi_s^{i-1}\rangle.
\]
For a.e.\ \(s\), condition on \(\mathscr G_s^{i-1}\). Applying the
variational inequality for relative entropy to \(\nu_s^i\) and
\(\bar\rho_s\), with test functions
\(\pm\lambda\widetilde\xi_s^{i-1}\), gives, for every \(\lambda>0\),
\[
\pm\langle\nu_s^i,\widetilde\xi_s^{i-1}\rangle
\le
\lambda^{-1}H(\nu_s^i\mid\bar\rho_s)
+
\lambda^{-1}
\log\int e^{\pm\lambda\widetilde\xi_s^{i-1}(x)}\bar\rho_s(\dd x).
\]
Since \(\langle\bar\rho_s,\widetilde\xi_s^{i-1}\rangle=0\), both logarithmic
terms are nonnegative by Jensen's inequality. Combining the two signs gives
\[
\begin{aligned}
\left|\langle\nu_s^i,\widetilde\xi_s^{i-1}\rangle\right|
&\le
\lambda^{-1}H(\nu_s^i\mid\bar\rho_s)
 +\lambda^{-1}
 \log\int e^{\lambda\widetilde\xi_s^{i-1}(x)}\bar\rho_s(\dd x)\\
&\quad
 +\lambda^{-1}
 \log\int e^{-\lambda\widetilde\xi_s^{i-1}(x)}\bar\rho_s(\dd x).
\end{aligned}
\]
Taking expectations, using
Lemma~\ref{lem:paper1-estimates}\textup{(iii)} and
Lemma~\ref{lem:conditional-replacement-estimates} with \(\lambda\) and
\(-\lambda\), and then choosing \(\lambda=1\), yields
\[
\mathbb E
\left|
\langle\theta_s^i,F_\varphi[\zeta_s^{i-1}]\rangle
\right|
\le C_{\varphi,T}(i-1)^{-1}.
\]
Hence, Lemma~\ref{lem:polylog-summability}, with \(\ell=n\), gives
\[
\mathbb E\int_0^T |C_s^{N,n}|\,\dd s
\le
\frac{C_{\varphi,T}}{\sqrt N}
\sum_{i=2}^N\frac{h_n(i/N)}{i-1}
\xrightarrow[N\to\infty]{}0.
\]

\smallskip
\noindent\emph{The term \(D\).}
Since
\[
\left|
\langle\theta_s^i,F_\varphi[\omega_s^{i-1}]\rangle
\right|
\le
C_\varphi
\|\theta_s^i\|_{\mathrm{TV}}
\|\omega_s^{i-1}\|_{\mathrm{TV}},
\]
the Cauchy--Schwarz inequality and
Lemma~\ref{lem:paper1-estimates}\textup{(iii)} give
\[
\mathbb E
\left|
\langle\theta_s^i,F_\varphi[\omega_s^{i-1}]\rangle
\right|
\le
C_{\varphi,T}\frac1{i-1}.
\]
The same logarithmic summability estimate gives
\[
\mathbb E\int_0^T |D_s^{N,n}|\,\dd s
\le
\frac{C_{\varphi,T}}{\sqrt N}
\sum_{i=2}^N
\frac{h_n(i/N)}{i-1}
\xrightarrow[N\to\infty]{}0.
\]

\smallskip
\noindent\emph{Conclusion.}
Combining the four estimates gives
\[
\mathbb E\int_0^T
\left(
|A_s^{N,n}|+|B_s^{N,n}|+|C_s^{N,n}|+|D_s^{N,n}|
\right)\,\dd s
\xrightarrow[N\to\infty]{}0.
\]
Since
\[
\sup_{0\le t\le T}
\left|\int_0^t \mathcal R_s^{N,n}(\varphi)\,\dd s\right|
\le
\int_0^T
\left(
|A_s^{N,n}|+|B_s^{N,n}|+|C_s^{N,n}|+|D_s^{N,n}|
\right)\,\dd s,
\]
this completes the proof.
\end{proof}

\subsection{Identification of subsequential limits}

After removing the nonlinear remainder with
Proposition~\ref{prop:weighted-conditional-measure-replacement}, we identify
the full interaction drift and pass the finite-\(N\) equation to the limit.

\begin{proposition}\label{prop:subsequential-limit-equation}
Let
\[
\bigl((Y^n)_{n\ge0},(\mathcal M^n)_{n\ge0}\bigr)
\]
be an \(E^{\mathbb N_0}\times E^{\mathbb N_0}\)-valued random element whose
law is a subsequential weak limit of the laws in
Proposition~\ref{prop:weighted-tightness}. Then, almost surely, for every
\(n\ge0\), every \(t\in[0,T]\), and every
\(\varphi\in C_c^\infty(\R^d)\),
\begin{align}
\langle Y_t^n,\varphi\rangle
&=
\langle Y_0^n,\varphi\rangle
+
\int_0^t
\langle Y_s^n,L_s\varphi\rangle\,\dd s\notag\\
&\quad+
\int_0^t
\left\langle
\bar\rho_s,
(K*Y_s^{n+1})\cdot\nabla\varphi
\right\rangle\,\dd s
+
M_t^n(\varphi),
\label{eq:subsequential-limit-equation}
\end{align}
where \(M_t^n(\varphi):=\langle\mathcal M_t^n,\varphi\rangle\).
With respect to the usual augmentation of the natural filtration generated by
\((Y,\mathcal M)\), each \(M^n(\varphi)\) is a continuous martingale.
\end{proposition}

\begin{proof}
By the definition of the subsequential limit, choose \(N_j\to\infty\) such
that the joint laws of
\(\bigl((Y^{N_j,n})_{n\ge0},(\mathcal M^{N_j,n})_{n\ge0}\bigr)\)
converge weakly to the law of
\(\bigl((Y^n)_{n\ge0},(\mathcal M^n)_{n\ge0}\bigr)\). Apply
Corollary~\ref{cor:subsequential-skorokhod} to \((N_j)_{j\ge1}\). On a new
stochastic basis,
\[
\left(
(\widetilde Y^{N_j,n})_{n\ge0},
(\widetilde{\mathcal M}^{N_j,n})_{n\ge0}
\right)
\longrightarrow
\left(
(\widetilde Y^n)_{n\ge0},
(\widetilde{\mathcal M}^n)_{n\ge0}
\right)
\]
almost surely in \(E^{\mathbb N_0}\times E^{\mathbb N_0}\). The limiting
random element has the same law as
\(\bigl((Y^n)_{n\ge0},(\mathcal M^n)_{n\ge0}\bigr)\), so it is enough to prove
the equation for the tilded processes.

Fix \(n\ge0\) and \(\varphi\in C_c^\infty(\R^d)\). We pass to the limit
term by term in \eqref{eq:finiteN-martingale-formulation}. Since
\(I^{N,n}=I_0^{N,n}+\mathcal R^{N,n}\),
\eqref{eq:leading-interaction-approximation},
Proposition~\ref{prop:weighted-conditional-measure-replacement}, and
\eqref{eq:first-particle-error-vanishes} give
\[
\mathbb E\sup_{t\le T}
\left|
\int_0^t
\left(I_s^{N,n}(\varphi)+\mathcal B_s^{N,n}(\varphi)\right)\,\dd s
-
\int_0^t
\langle Y_s^{N,n+1},\Gamma_s^\varphi\rangle\,\dd s
\right|
\longrightarrow0.
\]
Thus the interaction and boundary terms in
\eqref{eq:finiteN-martingale-formulation} may be replaced by the last
integral, with an error converging to zero in
\(L^1(\Omega;C([0,T];\R))\). By equality of laws, the same asymptotic
identity holds for the realized coordinates.

The convergence in \eqref{eq:skorokhod-countable-convergence} directly gives
the limits of the field and martingale terms. Since
\(s\mapsto L_s\varphi\) is continuous in \(H^{\beta,m}\), it also gives
\[
\int_0^\cdot
\langle\widetilde Y_s^{N_j,n},L_s\varphi\rangle\,\dd s
\longrightarrow
\int_0^\cdot
\langle\widetilde Y_s^n,L_s\varphi\rangle\,\dd s
\quad\text{almost surely in }C([0,T];\R).
\]
The interaction functional is the only term that requires localization.
Applying the localization argument from the proof of
Lemma~\ref{lem:leading-limit} to
\eqref{eq:skorokhod-countable-convergence} gives
\[
\int_0^\cdot
\langle\widetilde Y_s^{N_j,n+1},\Gamma_s^\varphi\rangle\,\dd s
\longrightarrow
\int_0^\cdot
\langle\widetilde Y_s^{n+1},\Gamma_s^\varphi\rangle\,\dd s
\quad\text{in probability in }C([0,T];\R).
\]
Passing to the limit in \eqref{eq:finiteN-martingale-formulation} and using
\[
\langle\gamma,\Gamma_s^\varphi\rangle
=
\left\langle
\bar\rho_s,(K*\gamma)\cdot\nabla\varphi
\right\rangle
\]
proves \eqref{eq:subsequential-limit-equation} for the fixed \(n\) and
\(\varphi\), almost surely for all \(t\in[0,T]\). Since \(n\) ranges over a
countable set, \(C_c^\infty(\R^d)\) is separable, and every term in
\eqref{eq:subsequential-limit-equation} is continuous in \(\varphi\), the
identity holds almost surely for all \(n\ge0\) and all
\(\varphi\in C_c^\infty(\R^d)\).

Finally, let \((\widetilde{\mathcal F}_t)_{t\in[0,T]}\) be the usual
augmentation of the natural filtration generated by
\((\widetilde Y,\widetilde{\mathcal M})\). The limiting processes are then
adapted. For \(0\le s<t\), any bounded continuous cylinder functional \(F\) of
the coordinates up to time \(s\), \(n\ge0\), and
\(\varphi\in C_c^\infty(\R^d)\), equality of laws, the finite-\(N\) martingale
property, \eqref{eq:skorokhod-countable-convergence}, and the uniform
\(L^2\)-bounds from the proof of Proposition~\ref{prop:weighted-tightness} give
\[
\begin{aligned}
&\widetilde{\mathbb E}\!\left[
\bigl(\widetilde M_t^n(\varphi)-\widetilde M_s^n(\varphi)\bigr)
F(\widetilde Y_{\cdot\wedge s},\widetilde{\mathcal M}_{\cdot\wedge s})
\right]\\
&\quad=
\lim_{j\to\infty}\widetilde{\mathbb E}\!\left[
\bigl(\widetilde M_t^{N_j,n}(\varphi)-
\widetilde M_s^{N_j,n}(\varphi)\bigr)
F(\widetilde Y^{N_j}_{\cdot\wedge s},
\widetilde{\mathcal M}^{N_j}_{\cdot\wedge s})
\right]
=0.
\end{aligned}
\]
A monotone-class argument shows that \(\widetilde M^n(\varphi)\) is an
\((\widetilde{\mathcal F}_t)\)-martingale.
This completes the proof.
\end{proof}

\section{Uniqueness and proof of the main result}\label{sec:uniqueness-main-result}

In this section, we prove uniqueness for the limiting hierarchy and complete
the proof of the main result. We use the mild form of the homogeneous
hierarchy to prove pathwise uniqueness and then apply the Yamada--Watanabe
theorem to obtain a strong solution. Finally, uniqueness implies convergence
of the full sequence.

Consider probabilistically weak solutions in the sense of
Definition~\ref{def:log-hierarchy-solution}, with the initial law prescribed in
Theorem~\ref{thm:log-hierarchy}.

\begin{proposition}\label{prop:weighted-uniqueness}
Pathwise uniqueness holds for probabilistically weak solutions of the limiting
logarithmic hierarchy in the sense of
Definition~\ref{def:log-hierarchy-solution}. Furthermore, the limiting
logarithmic hierarchy is strongly well posed.
\end{proposition}

\begin{proof}
The pathwise uniqueness argument is standard: because the noise is additive,
subtracting two solutions reduces the problem to uniqueness for a deterministic
homogeneous linear hierarchy with zero initial data. See also
\cite[Lemma~3.11]{wang2023gaussian} for the analogous argument for a linear
fluctuation equation.
Let \((Y^{(1),n})_{n\ge0}\) and \((Y^{(2),n})_{n\ge0}\) be two such solutions
on the same stochastic basis, with the same initial sequence and the same
martingale \(\mathcal M\), and set
\[
Z_t^n:=Y_t^{(1),n}-Y_t^{(2),n}.
\]
Then \(Z_0^n=0\), the common additive stochastic noise cancels, and for every
\(n\ge0\),
\begin{equation*}
\langle Z_t^n,\varphi\rangle
=
\int_0^t
\langle Z_s^n,L_s\varphi\rangle\,\dd s
+
\int_0^t
\left\langle
\bar\rho_s,
(K*Z_s^{n+1})\cdot\nabla\varphi
\right\rangle\,\dd s.
\end{equation*}

Let \(P_t=e^{t\Delta/2}\) be the heat semigroup and set
\(V_s=b(s,\cdot)+K*\bar\rho_s\). Duhamel's formula and \(Z_0^n=0\) give,
for \(t\in[0,T]\),
\[
Z_t^n
=
-\int_0^t
\nabla\cdot P_{t-s}
\left(
V_sZ_s^n+\bar\rho_s(K*Z_s^{n+1})
\right)\,\dd s
\]

Write
\[
\|\cdot\|_\ast:=\|\cdot\|_{H^{-\beta_\ast}}.
\]
Let \(q_1,q_2\in(0,1/4)\) be the parameters in
\eqref{eq:solution-class-summability} for the two solutions, choose
\(q<\min\{q_1,q_2\}\), and set
\[
\mathcal Z(t)
:=
\left(
\sum_{n=0}^\infty q^n\|Z_t^n\|_\ast^2
\right)^{1/2}.
\]
Then \(\sup_{t\le T}\mathcal Z(t)<\infty\) almost surely.
{
The multiplier estimate \eqref{eq:uniform-V-multiplier-bound} gives
\[
\|V_s z\|_\ast
\le C\|z\|_\ast,
\qquad z\in H^{-\beta_\ast}.
\]
To estimate the second term, let
\(\psi=(\psi_1,\ldots,\psi_d)\in
H^{\beta_\ast}(\R^d;\R^d)\). By the definition of the convolution and
Sobolev duality,
\[
\begin{aligned}
\left|
\left\langle \bar\rho_s(K*z),\psi\right\rangle
\right|
&=
\left|
\left\langle
z,\sum_{a=1}^d K_a^\vee*(\psi_a\bar\rho_s)
\right\rangle
\right|\\
&\le
\|z\|_\ast
\left\|
\sum_{a=1}^d K_a^\vee*(\psi_a\bar\rho_s)
\right\|_{H^{\beta_\ast}}\\
&\le
C\|K\|_{H^{\beta_\ast}}\|z\|_\ast\|\psi\|_\infty\\
&\le
C\|K\|_{H^{\beta_\ast}}\|z\|_\ast
\|\psi\|_{H^{\beta_\ast}}.
\end{aligned}
\]
The third line follows from
\eqref{eq:kernel-measure-convolution-bound}, and the last line follows from
the embedding \(H^{\beta_\ast}\hookrightarrow L^\infty\). Taking the
supremum over \(\|\psi\|_{H^{\beta_\ast}}\le1\) gives
\[
\|\bar\rho_s(K*z)\|_\ast
\le C\|K\|_{H^{\beta_\ast}}\|z\|_\ast.
\]
}
The heat-semigroup estimate gives, for vector-valued
\(f\in H^{-\beta_\ast}\) and \(\tau>0\),
\[
\|\nabla\cdot P_\tau f\|_\ast
\le C\tau^{-1/2}\|f\|_\ast.
\]
Taking the weighted \(\ell^2\)-norm in the mild equation and applying
Minkowski's inequality, we obtain
\[
\mathcal Z(t)
\le
C\int_0^t(t-s)^{-1/2}\mathcal Z(s)\,\dd s.
\]
Since \(\mathcal Z\) is bounded on \([0,T]\), the Volterra-type Gronwall
inequality \cite[Lemma~2.2]{zhang2010stochastic}, applied pathwise, yields
\(\mathcal Z\equiv0\).
Thus \(Z^n\equiv0\) for every \(n\ge0\), and pathwise uniqueness holds in the
solution class of
Definition~\ref{def:log-hierarchy-solution}.

Propositions~\ref{prop:weighted-tightness},
\ref{prop:gaussian-input-convergence}, and
\ref{prop:subsequential-limit-equation}, together with
Corollary~\ref{cor:subsequential-skorokhod}\textup{(iii)}, give weak existence.
Kurtz's Yamada--Watanabe theorem
\cite[Theorem~1.5]{kurtz2014yamada} then yields a probabilistically strong
solution and joint uniqueness in law. Hence the limiting logarithmic hierarchy
is strongly well posed.
\end{proof}

\begin{proof}[Proof of Theorem~\ref{thm:log-hierarchy}]
Combining Proposition~\ref{prop:weighted-tightness},
Corollary~\ref{cor:subsequential-skorokhod}\textup{(iii)}, and
Propositions~\ref{prop:gaussian-input-convergence} and
\ref{prop:subsequential-limit-equation}, every subsequence of
\((Y^{N,n})_{n\ge0}\) has a further subsequence converging in distribution in
\(E^{\mathbb N_0}\) to a probabilistically weak solution of the limiting
hierarchy with the prescribed initial law. By
Proposition~\ref{prop:weighted-uniqueness}, this limiting law is unique.
Consequently, the whole sequence converges in distribution.
\end{proof}

\section*{Acknowledgements}
This work was supported by the National Key R\&D Program of China
(2024YFA1015500) and the NSFC (12595282 and 12171009).

\bibliographystyle{plain}
\bibliography{CLT_sequential}

\bigskip
\begin{flushleft}
\small
\textsc{Zhenfu Wang}\\
Beijing International Center for Mathematical Research, Peking University,\\
5 Yiheyuan Road, Beijing 100871, China\\
\textit{Email address:}
\href{mailto:zwang@bicmr.pku.edu.cn}{\texttt{zwang@bicmr.pku.edu.cn}}

\medskip
\textsc{Xianliang Zhao}\\
Beijing International Center for Mathematical Research, Peking University,\\
5 Yiheyuan Road, Beijing 100871, China\\
\textit{Email address:}
\href{mailto:xzhaomath@gmail.com}{\texttt{xzhaomath@gmail.com}}
\end{flushleft}

\end{document}